\documentclass[12pt, reqno]{amsart}
\usepackage{amsmath, amsthm, amscd, amsfonts, amssymb, graphicx, color}
\usepackage[bookmarksnumbered, colorlinks, plainpages]{hyperref}
\usepackage{mathrsfs}
\usepackage{epsfig}
\usepackage{makecell}
\usepackage{float}
\usepackage{epstopdf}
\usepackage{subfigure}
\usepackage{mathbbol}
\usepackage{pifont}

\hypersetup{colorlinks=true,linkcolor=red,
anchorcolor=green, citecolor=cyan, urlcolor=red,
filecolor=magenta, pdftoolbar=true}

\newtheorem{theorem}{Theorem}[section]
\newtheorem{lemma}[theorem]{Lemma}
\newtheorem{proposition}[theorem]{Proposition}
\newtheorem{corollary}[theorem]{Corollary}
\theoremstyle{definition}

\newtheorem{example}[theorem]{Example}

\newtheorem{procedure}[theorem]{Procedure}

\theoremstyle{remark}
\newtheorem{remark}[theorem]{Remark}
\numberwithin{equation}{section}

\newcommand{\B}{\boldsymbol B}

\newcommand{\bq}{\begin{equation}}
\newcommand{\eq}{\end{equation}}
\newcommand{\beqn}{\begin{eqnarray*}}
\newcommand{\eeqn}{\end{eqnarray*}}
\newcommand{\beq}{\begin{eqnarray}}
\newcommand{\eeq}{\end{eqnarray}}

\newcommand{\rar}{\rightarrow}

\newcommand{\bc}{\begin{centre}}
\newcommand{\ec}{\end{centre}}

\newcommand{\ba}{\begin{array}}
\newcommand{\ea}{\end{array}}

\def \C{\mathbb{C}}

\newcommand*{\borel}[1]{{\mathfrak B}(#1)}
\newcommand*{\card}[1]{\mathrm{card}\,#1}
\newcommand*{\child}[1]{\mathsf{Chi}(#1)}

\newcommand*{\childn}[2]{{\mathsf{Chi}}^{\langle#1\rangle}(#2)}

\newcommand*{\des}[1]{{{\mathsf{Des}}(#1)}}

\newcommand*{\Ge}{\geqslant}
\newcommand*{\hh}{\mathcal{H}}
\newcommand*{\kk}{\mathcal{K}}
\newcommand*{\lambdab}{\boldsymbol\lambda}
\newcommand*{\Le}{\leqslant}
\newcommand*{\nbb}{\mathbb N}
\newcommand*{\parent}[1]{\mathsf{par}(#1)}

\newcommand*{\ran}{\mathrm{ran\,}}
\newcommand*{\rbb}{\mathbb R}
\newcommand*{\rot}{\mathsf{\omega}}
\newcommand*{\slam}{S_{\lambdab}}

\newcommand*{\supp}[1]{\mathrm{supp}(#1)}
\newcommand*{\tcal}{\mathscr T}
\newcommand*{\zbb}{\mathbb Z}

\begin{document}

\setcounter{page}{1}

\title[Complete systems of unitary invariants for $2$-isometries]
{Complete systems of unitary invariants for \\ some
classes of $2$-isometries}

\author[A. Anand, S. Chavan, Z.\ J.\ Jab{\l}o\'nski, \MakeLowercase{and} J. Stochel]
{Akash Anand,$^1$ Sameer Chavan,$^1$ Zenon Jan
Jab{\l}o\'nski,$^2$ \\ \MakeLowercase{and} Jan
Stochel$^{2*}$}

\address{$^{1}$Department of Mathematics and Statistics\\
Indian Institute of Technology Kanpur, India.}
\email{\textcolor[rgb]{0.00,0.00,0.84}{akasha@iitk.ac.in;
chavan@iitk.ac.in}}
\address{$^{2}$Instytut Matematyki, Uniwersytet Jagiello\'nski,
ul.\ \L ojasiewicza 6, PL-30348 Kra\-k\'ow, Poland.}
\email{\textcolor[rgb]{0.00,0.00,0.84}{Zenon.Jablonski@im.uj.edu.pl;
Jan.Stochel@im.uj.edu.pl}}

\dedicatory{Dedicated to the memory of Professor
Ronald G. Douglas}





\begin{abstract}
The unitary equivalence of $2$-isometric operators
satisfying the so-called kernel condition is
characterized. It relies on a model for such operators
built on operator valued unilateral weighted shifts
and on a characterization of the unitary equivalence
of operator valued unilateral weighted shifts in a
fairly general context. A complete system of unitary
invariants for $2$-isometric weighted shifts on rooted
directed trees satisfying the kernel condition is
provided. It is formulated purely in the langauge of
graph-theory, namely in terms of certain generation
branching degrees. The membership of the Cauchy dual
operators of $2$-isometries in classes $C_{0 \cdot}$
and $C_{\cdot 0}$ is also studied.
\end{abstract}

\maketitle

   \section{Introduction}
We begin by defining the basic concepts discussed in
this paper. Let $\hh$ be a (complex) Hilbert space and
$\B(\hh)$ stand for the $C^*$-algebra of all bounded
linear operators on $\hh$. We say that an operator
$T\in \B(\hh)$ is
   \begin{enumerate}
   \item[$\bullet$] {\em hyponormal} if $T^*T-TT^* \Ge 0,$
   \item[$\bullet$] {\em subnormal} if it
has a normal extension in a possibly larger Hilbert
space,
   \item[$\bullet$] {\em $2$-hyperexpansive} if
$I - 2 T^*T + T^{*2}T^2 \Le 0,$
   \item[$\bullet$] {\em
$2$-isometric} if $I - 2 T^*T + T^{*2}T^2 =0.$
   \end{enumerate}
Subnormal operators are hyponormal (see
\cite[Proposition~ II.4.2]{Co}) and $2$-isome\-tries
are $2$-hyperexpansive, but none of these implications
can be reversed (see \cite[Exercise~ 3, p.\ 50]{Co}
and \cite[Lemma~ 6.1]{Ja-St}, respectively). Moreover,
hyponormal operators which are $2$-hyperexpansive are
isometric (see \cite[Theorem~ 3.4]{Ja-St}). The theory
of subnormal and hyponormal operators was initiated by
Halmos \cite{Hal-s}. The notion of a $2$-isometry was
invented by Agler \cite{Ag-0}, while the concept of a
$2$-hyperexpansive operator goes back to Richter
\cite{R-0} (see also \cite[Remark~ 2]{At}). The {\it
Cauchy dual operator} $T'$ of a left-invertible
operator $T$ is defined by $T'=T(T^*T)^{-1}.$ This
concept is due to Shimorin \cite{Sh}. The basic
relationship between $2$-hyperexpansions and
hyponormal operators via the Cauchy dual transform is
as follows (see \cite[Sect.\ 5]{Sh-1} and
\cite[Theorem~ 2.9]{Ch-0}).
   \begin{align} \label{2hypcon}
   \begin{minipage}{70ex}
{\em If $T\in \B(\hh)$ is a $2$-hyperexpansive
operator, then $T$ is left-invertible and $T'$ is a
hyponormal contraction.}
   \end{minipage}
   \end{align}

In a recent paper \cite{A-C-J-S}, the present authors
solved the Cauchy dual subnormality problem in the
negative by showing that there are $2$-isometric
operators $T$ whose Cauchy dual operators $T'$ are not
subnormal. One of the ideas of constructing such
counterexamples relies on perturbing the so-called
kernel condition in the context of weighted shifts on
directed trees (see \cite{JJS} for more information on
this class of operators). Recall from \cite{A-C-J-S}
that $T\in \B(\hh)$ satisfies the {\em kernel
condition}~ if
   \begin{align}  \label{kc}
T^*T (\ker T^*) \subseteq \ker T^*.
   \end{align}
It was proved in \cite[Theorem~ 6.5]{A-C-J-S} that if
$\tcal$ is a rooted directed tree and $\slam$ is a
$2$-isometric weighted shift on $\tcal$ with nonzero
weights which satisfies the perturbed kernel
condition, then the Cauchy dual operator $\slam'$ of
$\slam$ is subnormal if and only if $\slam$ satisfies
the kernel condition. Further, it was shown in
\cite[Theorem~ 3.3]{A-C-J-S} that the Cauchy dual
operator $T'$ of a $2$-isometry $T$ satisfying the
kernel condition is always subnormal. This can in turn
be derived from a model theorem for $2$-isometries
satisfying the kernel condition (see \cite[Theorem~
2.5]{A-C-J-S}). The model itself is built on operator
valued unilateral weighted shifts and is the starting
point of the present investigations. It is worth
mentioning that there are Dirichlet-type models for
cyclic analytic $2$-isometries and for finitely
multicyclic $2$-isometries given by Richter
\cite[Theorem~ 5.1]{R-1} and by Agler and Stankus
\cite[Theorem~ 3.49]{Ag-St}, respectively. Richter
used his model to characterize unitary equivalence of
cyclic analytic $2$-isometries (see \cite[Theorem~
5.2]{R-1}). As far as we know, there are no models for
arbitrary $2$-isometries.

The paper is organized as follows. In Section
\ref{Sec4}, looking for a complete system of unitary
invariants for $2$-isometries satisfying the kernel
condition, we first discuss the question of unitary
equivalence of operator valued unilateral weighted
shifts in the general context. This class of operators
was investigated by Lambert \cite{lam-1}. An essential
progress in their study, also relevant for our present
work, was done in \cite{Ja-3}. As opposed to the
previous approaches, our do not require the operator
weights to be even quasi-invertible. We only assume
that they have dense range. We provide a
characterization of unitary equivalence of such
operators (see Theorem~ \ref{przepluni}). Under some
carefully chosen constraints, we obtain a
characterization of their unitary equivalence (see
Theorem~ \ref{Fran4}), which resembles that for scalar
weighted shifts (cf.\ \cite[Theorem~ 1]{Shi}). We
conclude this section by characterizing the unitary
equivalence of orthogonal sums (of arbitrary
cardinality) of injective unilateral weighted shifts
(see Theorem~ \ref{desz1}). We want to draw the
reader's attention to \cite{Ba-Mi}, where the
so-called block shifts generalizing operator valued
unilateral weighted shifts were studied.

In Section \ref{Sec5}, using the model for
$2$-isometries satisfying the kernel condition (see
\cite[Theorem~ 2.5]{A-C-J-S}), we answer the question
of when two such operators are unitarily equivalent
(see Theorem~ \ref{fincyc} and Lemma~ \ref{unrown}).
We also answer the question of when a completely
non-unitary $2$-isometry satisfying the kernel
condition is unitarily equivalent to an orthogonal sum
of scalar unilateral weighted shifts (see Theorem~
\ref{fincyc2}). This enables us to show that each
finitely multicyclic completely non-unitary
$2$-isometry satisfying the kernel condition is a
finite orthogonal sum of weighted shifts (see
Corollary~ \ref{mulcyc}). As a consequence, the
adjoint of any such operator is in the Cowen-Douglas
class (see \cite[Corollary~ 3.7]{Cha2008} for a more
general result). We refer the reader to \cite{C-D} for
the definition of the Cowen-Douglas class.

In Section \ref{Sec9}, we investigate $2$-isometric
weighted shifts on directed trees satisfying the
condition \eqref{hypo+}, which in general is stronger
than the kernel condition. However, they coincide in
the case when the directed tree is leafless and the
weights of the weighted shift under consideration are
nonzero (see \cite[Lemma~ 5.6]{A-C-J-S}). Example~
\ref{obustr} shows that the fact that a weighted shift
on a rooted directed tree is completely non-unitary
(see \cite[Lemma~ 5.3(viii)]{A-C-J-S}) is no longer
true for weighted shifts on rootless directed trees
even though they are isometric and non-unitary.
Theorem~ \ref{2isscs-t} provides a model for
$2$-isometric weighted shifts on rooted directed trees
that satisfy the condition \eqref{hypo+}. These
operators are modelled by orthogonal sums of
inflations of unilateral weighted shifts whose weights
come from a single $2$-isometric unilateral weighted
shift. What is more, the additive exponent of the
$k$th inflation that appears in the orthogonal
decomposition \eqref{zenob} is equal to ${\mathfrak
j}^{\tcal}_k,$ the $k$th generation branching degree
of the underlying graph $\tcal$. This enables us to
answer the question of when two such operators are
unitarily equivalent by using ${\mathfrak
j}^{\tcal}_k$ (see Theorem~ \ref{equival}). We
conclude this section by showing that there are two
unitarily equivalent $2$-isometric weighted shifts on
non-graph isomorphic directed trees with nonzero
weights which satisfy the kernel condition (see
Example~ \ref{2+3}).

In Section \ref{Sec6}, we continue our investigations
of unitary invariants. We begin by calculating
explicitly another unitary invariant, namely the SOT
limit $\mathsf A_{T'}$ of the sequence
$\{T'^{*n}T'^{n}\}_{n=1}^{\infty}$ for two classes of
$2$-isometries $T$ (see Lemma~ \ref{convcd}). We next
show that the Cauchy dual operator $T^\prime$ of a
$2$-isometry $T$ is of class $C_{\cdot 0}$ if and only
if $T$ is completely non-unitary. Under the additional
assumption that $T$ satisfies the kernel condition,
the Cauchy dual operator $T^\prime$ is of class
$C_{0\cdot}$ if and only if $G(\{1\})=0,$ or
equivalently if and only if $E(\{1\})=0,$ where $G$
and $E$ are the spectral measures of $T^*T$ and the
zeroth weight $W_0$ of the model operator $W$ for $T,$
respectively (see Theorem~ \ref{coo}). Note that
non-isometric quasi-Brownian isometries do not satisfy
the kernel condition (see \cite[Example~ 4.4 and
Corollary~ 4.6]{A-C-J-S}) and their Cauchy dual
operators are never of class $C_{0\cdot}$ (see
Proposition~ \ref{coo-qB}(i)).

Now we fix notation and terminology. Let $\C$ stand
for the set of complex numbers. Denote by $\nbb$,
$\zbb_+$ and $\rbb_+$ the sets of positive integers,
nonnegative integers and nonnegative real numbers,
respectively. Given a set $X$, we write $\card{X}$ for
the cardinality of $X$ and denote by
$\chi_{\varDelta}$ the characteristic function of a
subset $\varDelta$ of $X$. The $\sigma$-algebra of all
Borel subsets of a topological space $X$ is denoted by
$\borel{X}$. In this paper, Hilbert spaces are assumed
to be complex and operators are assumed to be linear.
Let $\hh$ be a Hilbert space. As usual, we denote by
$\dim \hh$ the orthogonal dimension of $\hh$. If $f
\in \hh$, then $\langle f \rangle$ stands for the
linear span of the singleton of $f$. Given another
Hilbert space $\kk$, we denote by $\B(\hh,\kk)$ the
Banach space of all bounded operators from $\hh$ to
$\kk$. The kernel, the range and the modulus of an
operator $T \in \B(\hh,\kk)$ are denoted by $\ker T,$
$\ran T$ and $|T|,$ respectively. We abbreviate
$\B(\hh,\hh)$ to $\B(\hh)$ and regard $\B(\hh)$ as a
$C^*$-algebra. Its unit, which is the identity
operator on $\hh$, is denoted here by $I_\hh,$ or
simply by $I$ if no ambiguity arises. We write
$\sigma(T)$ for the spectrum of $T\in \B(\hh).$ Given
$T \in \B(\hh)$ and a cardinal number $\mathfrak n$,
we set $\hh^{\oplus{\mathfrak n}} = \bigoplus_{j\in J}
\hh_j$ and $T^{\oplus{\mathfrak n}}= \bigoplus_{j\in
J} T_j$ with $\hh_j=\hh$ and $T_j=T$ for all $j\in J$,
where $J$ is an index set of cardinality $\mathfrak
n$. We call $\hh^{\oplus{\mathfrak n}}$ and
$T^{\oplus{\mathfrak n}}$ the {\em $\mathfrak n$-fold
inflation} of $\hh$ and $T$, respectively. We adhere
to the convention that $\hh^{\oplus{0}}=\{0\}$ and
$T^{\oplus{0}}=0$. If $S$ and $T$ are Hilbert space
operators which are unitarily equivalent, then we
write $S \cong T$.

We say that an operator $T\in \B(\hh)$ is {\em
completely non-unitary} (resp., {\em pure}) if there
is no nonzero reducing closed vector subspace
$\mathcal L$ of $\hh$ such that the restriction
$T|_{\mathcal L}$ of $T$ to $\mathcal L$ is a unitary
(resp., a normal\/) operator. Following \cite{R-1}, we
call $T$ {\em analytic} if $\bigcap_{n=1}^{\infty}
T^n(\hh)=\{0\}$. Note that any analytic operator is
completely non-unitary. It is well known that any
operator $T\in \B(\hh)$ has a unique orthogonal
decomposition $T=N\oplus P$ such that $N$ is a normal
operator and $P$ is a pure operator (see
\cite[Corollary~ 1.3]{Mo}). We shall refer to $N$ and
$P$ as the {\em normal} and {\em pure} parts of $T$,
respectively. The following fact can be deduced from
\cite[Corollary~ 1.3]{Mo}.
   \begin{lemma}\label{unrown}
Operators $T_1\in \B(\hh_1)$ and $T_2\in \B(\hh_2)$
are unitarily equivalent if and only if their
corresponding normal and pure parts are unitarily
equivalent.
   \end{lemma}
   \section{\label{Sec4}Unitary equivalence of operator
valued unilateral weighted shifts} In this section,
the question of unitary equivalence of operator valued
unilateral weighted shifts is revisited. First, we
give a necessary and sufficient condition for two such
operators whose weights have dense range to be
unitarily equivalent (see Theorem~ \ref{przepluni}).
This result generalizes in particular \cite[Corollary~
3.3]{lam-1} in which weights are assumed to be
invertible. If weights are more regular, where the
regularity does not refer to invertibility, then the
characterization of unitary equivalence takes on a
much simpler form (see Theorem~ \ref{Fran4} and
Corollary~ \ref{dom2}). As an application, we answer
the question of when two orthogonal sums of uniformly
bounded families of injective unilateral weighted
shifts are unitarily equivalent (see Theorem~
\ref{desz1}).

We begin by proving a criterion for the modulus of a
finite product of bounded operators to be equal to the
product of their moduli.
   \begin{lemma} \label{kommod}
Let $n$ be an integer greater than or equal to $2$.
Suppose $A_1, \ldots, A_n \in \B(\hh)$ are such that
$|A_i|$ commutes with $A_j$ whenever $i < j$. Then
   \begin{enumerate}
   \item[(i)] the operators $|A_1|, \ldots, |A_n|$
mutually commute,
   \item[(ii)] $|A_1 \,\cdots\, A_n|^2 = |A_1|^2 \,\cdots\,
|A_n|^2$,
   \item[(iii)] $|A_1 \,\cdots\, A_n| = |A_1| \,\cdots\,
|A_n|$.
   \end{enumerate}
   \end{lemma}
   \begin{proof}
(i) Fix integers $i,j \in \{1, \ldots, n\}$ such that
$i<j$. Since $|A_i| A_j = A_j |A_i|$, and thus $|A_i|
A_j^* = A_j^* |A_i|$, we see that
$|A_i||A_j|^2=|A_j|^2|A_i|$. Hence
$|A_i||A_j|=|A_j||A_i|$, which proves (i).

(ii) By our assumption and (i), we have
\allowdisplaybreaks
   \begin{align}   \notag
|A_1\,\cdots\, A_n|^2 & = A_n^* \,\cdots\, A_2^*
|A_1|^2 A_2 \,\cdots\, A_n
   \\   \notag
& = |A_1|^2 A_n^* \,\cdots\, A_3^* |A_2|^2 A_3
\,\cdots\, A_n
   \\   \notag
& \hspace{4.5ex} \vdots
   \\  \label{Fran3}
& = |A_1|^2 \,\cdots\, |A_n|^2.
   \end{align}

(iii) It follows from \eqref{Fran3} and (i) that
   \begin{align*}
|A_1 \,\cdots\, A_n|^2 = (|A_1| \,\cdots\, |A_n|)^2.
   \end{align*}
Applying the square root theorem and the fact that the
product of commuting positive bounded operators is
positive, we conclude that (iii) holds.
   \end{proof}
Let us recall the definition of an operator valued
unilateral weighted shift. Suppose $\mathcal M$ is a
\underline{nonzero} Hilbert space. Denote by
$\ell^2_{\mathcal M}$ the Hilbert space of all vector
sequences $\{h_n\}_{n=0}^{\infty} \subseteq \mathcal
M$ such that $\sum_{n=0}^{\infty} \|h_n\|^2 < \infty$
equipped with the standard inner product
   \begin{align*}
\big\langle \{g_n\}_{n=0}^{\infty},
\{h_n\}_{n=0}^{\infty}\big\rangle =
\sum_{n=0}^{\infty} \langle g_n, h_n \rangle, \quad
\{g_n\}_{n=0}^{\infty}, \{h_n\}_{n=0}^{\infty} \in
\ell^2_{\mathcal M}.
   \end{align*}
Let $\{W_n\}_{n=0}^{\infty} \subseteq \B(\mathcal M)$
be a uniformly bounded sequence of operators. Then the
operator $W \in \B(\ell^2_{\mathcal M})$ defined by
   \begin{align*}
W(h_0, h_1, \ldots) = (0, W_0 h_0, W_1 h_1, \ldots),
\quad (h_0, h_1, \ldots) \in \ell^2_{\mathcal M},
   \end{align*}
is called an {\em operator valued unilateral weighted
shift} with weights $\{W_n\}_{n=0}^{\infty}$. It is
easy to verify that \allowdisplaybreaks
   \begin{align} \label{aopws}
W^*(h_0, h_1, \ldots) &= (W_0^*h_1, W_1^*h_2, \ldots),
\quad (h_0, h_1, \ldots) \in \ell^2_{\mathcal M},
   \\   \label{aopws2}
W^*W(h_0, h_1, \ldots) &= (W_0^*W_0h_0, W_1^*W_1h_1,
\ldots), \quad (h_0, h_1, \ldots) \in \ell^2_{\mathcal
M}.
   \end{align}
If each weight $W_n$ of $W$ is an invertible (resp., a
positive) element of the $C^*$-algebra $\B(\mathcal
M)$, then we say that $W$ is an operator valued
unilateral weighted shift with {\em invertible}
(resp., {\em positive}) weights. Putting $\mathcal
M=\C$, we arrive at the well-known notion of a
unilateral weighted shift in $\ell^2_{\C}=\ell^2$.

From now on, we assume that $\mathcal M^{(1)}$ and
$\mathcal M^{(2)}$ are nonzero Hilbert spaces and
$W^{(1)} \in \B(\ell^2_{\mathcal M^{(1)}})$ and
$W^{(2)} \in \B(\ell^2_{\mathcal M^{(2)}})$ are
operator valued unilateral weighted shifts with
weights $\{W_n^{(1)}\}_{n=0}^{\infty} \subseteq
\B(\mathcal M^{(1)})$ and
$\{W_n^{(2)}\}_{n=0}^{\infty} \subseteq \B(\mathcal
M^{(2)})$, respectively. Below, under the assumption
that the weights of $W^{(1)}$ have dense range, we
characterize bounded operators which intertwine
$W^{(1)}$ and $W^{(2)}$ (see \cite[Lemma~ 2.1]{lam-1}
for the case of invertible weights).
   \begin{lemma} \label{przepl}
Suppose that each operator $W_n^{(1)}$, $n \in
\zbb_+$, has dense range. Let $A\in
\B(\ell^2_{\mathcal M^{(1)}}, \ell^2_{\mathcal
M^{(2)}})$ be an operator with the matrix
representation $[A_{i,j}]_{i,j=0}^{\infty}$, where
$A_{i,j}\in\B(\mathcal M^{(1)}, \mathcal M^{(2)})$ for
all $i,j\in\zbb_+$. Then the following two conditions
are equivalent{\em :}
   \begin{enumerate}
   \item[(i)] $AW^{(1)} = W^{(2)}A$,
   \item[(ii)] $A$ is lower triangular, that is, $A_{i,j}=0$
whenever $i<j$, and
   \begin{align} \label{ABBA0}
A_{i,j} W_{j-1}^{(1)} \,\cdots\, W_{0}^{(1)} =
W_{i-1}^{(2)} \,\cdots\, W_{i-j}^{(2)} A_{i-j,0},
\quad i \Ge j \Ge 1.
   \end{align}
   \end{enumerate}
   \end{lemma}
   \begin{proof}
Denote by $\delta_{i,j}$ the Kronecker delta function.
Since $W^{(k)}$ has the matrix representation
$[\delta_{i,j+1} W_j^{(k)}]_{i,j=0}^{\infty}$ for
$k=1,2$, we see that (i) holds if and only if
$A_{i,j+1} W_j^{(1)} = W_{i-1}^{(2)} A_{i-1,j}$ for
all $i,j \in \zbb_+$ (with the convention that
$W_{-1}^{(2)}=0$ and $A_{-1,j}=0$ for $j\in \zbb_+$).
Hence, (i) holds if and only if the following
equations hold \allowdisplaybreaks
   \begin{align}  \label{ABBA1}
A_{0,j} & = 0, \quad j \in \nbb,
   \\  \label{ABBA3}
A_{i+1,j+1} W_j^{(1)} & = W_i^{(2)} A_{i,j}, \quad i,j
\in \zbb_+.
   \end{align}

(i)$\Rightarrow$(ii) By induction, we infer from
\eqref{ABBA3} that
   \begin{align} \label{ABBA2}
A_{i+k,j+k} W_{j+k-1}^{(1)} \,\cdots\, W_{j}^{(1)} =
W_{i+k-1}^{(2)} \,\cdots\, W_{i}^{(2)} A_{i,j}, \quad
i,j \in \zbb_+, \, k \in \nbb.
   \end{align}
This and \eqref{ABBA1} combined with the assumption
that each $W_n^{(1)}$ has dense range, imply that $A$
is lower triangular. It is a matter of routine to show
that \eqref{ABBA2} implies \eqref{ABBA0}.

(ii)$\Rightarrow$(i) Since $A$ is lower triangular and
\eqref{ABBA0} holds, it remains to show that
\eqref{ABBA3} is valid whenever $i\Ge j\Ge 1$.
Applying \eqref{ABBA0} again, we get
\allowdisplaybreaks
   \begin{align*}
A_{i+1,j+1} W_j^{(1)} \Big(W_{j-1}^{(1)} \,\cdots\,
W_0^{(1)}\Big) & = W_i^{(2)} \Big(W_{i-1}^{(2)}
\,\cdots\, W_{i-j}^{(2)} A_{i-j,0}\Big)
   \\
& = W_i^{(2)} A_{i,j} \Big(W_{j-1}^{(1)} \,\cdots\,
W_0^{(1)}\Big).
   \end{align*}
Since each operator $W_n^{(1)}$ has dense range, we
conclude that $A_{i+1,j+1} W_j^{(1)}=W_i^{(2)}
A_{i,j}$. This completes the proof.
   \end{proof}
The question of when the operators $W^{(1)}$ and
$W^{(2)}$ whose weights have dense range are unitarily
equivalent is answered by the following theorem (see
\cite[Corollary~ 3.3]{lam-1} for the case of
invertible weights).
   \begin{theorem} \label{przepluni}
Suppose that for any $k=1,2$ and every $n \in \zbb_+$,
the operator $W_n^{(k)}$ has dense range. Then the
following two conditions are equivalent{\em :}
   \begin{enumerate}
   \item[(i)] $W^{(1)} \cong  W^{(2)}$,
   \item[(ii)] there exists a unitary isomorphism $U_0\in
\B(\mathcal M^{(1)}, \mathcal M^{(2)})$ such that
   \begin{align} \label{Fra1}
|W_{[i]}^{(1)}| = U_0^*|W_{[i]}^{(2)}| U_0, \quad i
\in \nbb,
   \end{align}
where $W_{[i]}^{(k)} = W_{i-1}^{(k)} \,\cdots\,
W_{0}^{(k)}$ for $i\in \nbb$ and $k=1,2$.
   \end{enumerate}
   \end{theorem}
   \begin{proof}
(i)$\Rightarrow$(ii) Suppose that $U\in
\B(\ell^2_{\mathcal M^{(1)}}, \ell^2_{\mathcal
M^{(2)}})$ is a unitary isomorphism such that $U
W^{(1)} = W^{(2)}U$ and $[U_{i,j}]_{i,j=0}^{\infty}$
is the matrix representation of $U$, where
$\{U_{i,j}\}_{i,j=0}^{\infty} \subseteq \B(\mathcal
M^{(1)}, \mathcal M^{(2)})$. It follows from Lemma~
\ref{przepl} that the operator $U$ is lower
triangular. Since $U^*=U^{-1}$ is a unitary
isomorphism with the corresponding matrix
representation $[(U_{j,i})^*]_{i,j=0}^{\infty}$ and
$U^*W^{(2)} = W^{(1)}U^*$, we infer from Lemma~
\ref{przepl} that $U^*$ is lower triangular. In other
words, $U_{i,j}=0$ whenever $i\neq j$. Since $U$ is a
unitary isomorphism, we deduce that for any $i \in
\zbb_+$, $U_i:=U_{i,i}$ is a unitary isomorphism. It
follows from \eqref{ABBA0} that
   \begin{align*}
U_i W_{[i]}^{(1)} = W_{[i]}^{(2)} U_0, \quad i \in
\nbb.
   \end{align*}
This yields
   \begin{align*}
|W_{[i]}^{(1)}|^2 = (W_{[i]}^{(1)})^* U_i^* U_i
W_{[i]}^{(1)} = U_0^* |W_{[i]}^{(2)}|^2 U_0, \quad
i\in \nbb.
   \end{align*}
Applying the square root theorem implies \eqref{Fra1}.

(ii)$\Rightarrow$(i) In view of \eqref{Fra1}, we have
   \begin{align}  \label{fran2}
\|W_{[i]}^{(1)} f\| = \||W_{[i]}^{(1)}| f\| =
\||W_{[i]}^{(2)}| U_0 f\| = \|W_{[i]}^{(2)} U_0 f\|,
\quad f \in \mathcal M^{(1)}, \, i \in \nbb.
   \end{align}
By our assumption, for any $k=1,2$ and every $i\in
\nbb$, the operator $W_{[i]}^{(k)}$ has dense range.
Hence, by \eqref{fran2}, for every $i\in \nbb$, there
exists a unique unitary isomorphism $U_i\in
\B(\mathcal M^{(1)}, \mathcal M^{(2)})$ such that
   \begin{align*}
U_i W_{[i]}^{(1)} = W_{[i]}^{(2)} U_0, \quad i \in
\nbb.
   \end{align*}
Set $U=\bigoplus_{i=0}^{\infty} U_i$. Applying Lemma~
\ref{przepl} to $A=U$, we get $UW^{(1)} = W^{(2)}U$
which completes the proof.
   \end{proof}
Under additional assumptions on weights, the above
characterization of unitary equivalence of $W^{(1)}$
and $W^{(2)}$ can be substantially simplified.
   \begin{theorem} \label{Fran4}
Suppose that for any $k=1,2$ and every $n \in \zbb_+$,
$\ker W_n^{(1)} = \{0\}$, the operator $W_n^{(k)}$ has
dense range and $|W_n^{(k)}|$ commutes with
$W_m^{(k)}$ whenever $m < n$. Then the following two
conditions are equivalent{\em :}
   \begin{enumerate}
   \item[(i)] $W^{(1)} \cong  W^{(2)}$,
   \item[(ii)] there exists a unitary isomorphism $U_0\in
\B(\mathcal M^{(1)}, \mathcal M^{(2)})$ such that
   \begin{align} \label{Fran1}
|W_{n}^{(1)}| = U_0^*|W_{n}^{(2)}| U_0, \quad n\in
\zbb_+.
   \end{align}
   \end{enumerate}
   \end{theorem}
   \begin{proof}
(i)$\Rightarrow$(ii) It follows from Theorem~
\ref{przepluni} that there exists a unitary
isomorphism $U_0\in \B(\mathcal M^{(1)}, \mathcal
M^{(2)})$ such that \eqref{Fra1} holds. We will show
that \eqref{Fran1} is valid. The case of $n=0$ follows
directly from \eqref{Fra1} with $i=1$. Suppose now
that $n\in \nbb$. Then, by Lemma~ \ref{kommod} and
\eqref{Fra1}, we have \allowdisplaybreaks
   \begin{align}  \notag
|W_n^{(1)}| |W_{[n]}^{(1)}| = |W_{[n+1]}^{(1)}| & =
U_0^*|W_{[n+1]}^{(2)}| U_0
   \\ \notag
& = U_0^* |W_n^{(2)}| U_0 U_0^* |W_{[n]}^{(2)}| U_0
   \\ \label{dom1}
& = U_0^* |W_n^{(2)}| U_0 |W_{[n]}^{(1)}|.
   \end{align}
Since $W_{[n]}^{(1)}$ is injective, we deduce that the
operator $|W_{[n]}^{(1)}|$ has dense range. Hence, by
\eqref{dom1}, $|W_n^{(1)}| = U_0^* |W_n^{(2)}| U_0$.

(ii)$\Rightarrow$(i) It follows from Lemma~
\ref{kommod} that
   \begin{align*}
|W_{[i]}^{(k)}| = |W_{i-1}^{(k)}| \,\cdots\,
|W_{0}^{(k)}|, \quad i\in \nbb, \, k=1,2.
   \end{align*}
Hence, by \eqref{Fran1} and Lemma~ \ref{kommod}, we
have
   \begin{align*}
|W_{[i]}^{(1)}| = (U_0^*|W_{i-1}^{(2)}| U_0)
\,\cdots\, (U_0^* |W_{0}^{(2)}|U_0) = U_0^*
|W_{[i]}^{(2)}| U_0, \quad i\in \nbb.
   \end{align*}
In view of Theorem~ \ref{przepluni}, $W^{(1)} \cong
W^{(2)}$. This completes the proof.
   \end{proof}
   \begin{corollary} \label{dom2}
Suppose that for $k=1,2$,
$\{W_n^{(k)}\}_{n=0}^{\infty}$ are injective diagonal
operators with respect to the same orthonormal basis
of $\mathcal M^{(k)}$. Then $W^{(1)} \cong W^{(2)}$ if
and only if the condition {\em (ii)} of Theorem~ {\em
\ref{Fran4}} is satisfied.
   \end{corollary}
   \begin{remark}
First, it is easily verifiable that Theorem~
\ref{Fran4} remains true if instead of assuming that
the operators $\{W_n^{(1)}\}_{n=0}^{\infty}$ are
injective, we assume that the operators
$\{W_n^{(2)}\}_{n=0}^{\infty}$ are injective. Second,
the assumption that the operators
$\{W_n^{(1)}\}_{n=0}^{\infty}$ are injective was used
only in the proof of the implication
(i)$\Rightarrow$(ii) of Theorem~ \ref{Fran4}. Third,
the assertion (ii) of Theorem~ \ref{Fran4} implies
that the operators $\{W_n^{(1)}\}_{n=0}^{\infty}$ are
injective if and only if the operators
$\{W_n^{(2)}\}_{n=0}^{\infty}$ are injective. \hfill
$\diamondsuit$
   \end{remark}
We are now in a position to characterize the unitary
equivalence of two orthogonal sums of uniformly
bounded families of injective unilateral weighted
shifts.
   \begin{theorem} \label{desz1}
Suppose for $k=1,2$, $\varOmega_k$ is a nonempty set
and $\{S_{\omega}^{(k)}\}_{\omega \in \varOmega_k}
\subseteq \B(\ell^2)$ is a uniformly bounded family of
injective unilateral weighted shifts. Then the
following two conditions are equivalent{\em :}
   \begin{enumerate}
   \item[(i)] $\bigoplus_{\omega\in \varOmega_1}
S_{\omega}^{(1)} \cong \bigoplus_{\omega\in
\varOmega_2} S_{\omega}^{(2)}$,
   \item[(ii)] there exists a bijection $\varPhi\colon
\varOmega_1 \to \varOmega_2$ such that
$S_{\varPhi(\omega)}^{(2)}=S_{\omega}^{(1)}$ for all
$\omega \in \varOmega_1$.
   \end{enumerate}
   \end{theorem}
   \begin{proof}
(i)$\Rightarrow$(ii) For $k=1,2$, we denote by
$\hh^{(k)}$ the Hilbert space in which the orthogonal
sum $T^{(k)}:=\bigoplus_{\omega\in \varOmega_k}
S_{\omega}^{(k)}$ acts and choose an orthonormal basis
$\{e_{\omega, n}^{(k)}\}_{\omega \in \varOmega_k,n \in
\zbb_+}$ of $\hh^{(k)}$ such that $T^{(k)} e_{\omega,
n}^{(k)} = \lambda_{\omega, n}^{(k)} e_{\omega,
n+1}^{(k)}$ for all $\omega \in \varOmega_k$ and $n\in
\zbb_+$, where $\lambda_{\omega, n}^{(k)}$ are nonzero
complex numbers. Clearly, the space $\bigoplus_{n \in
\zbb_+} \langle e_{\omega, n}^{(k)} \rangle$ reduces
$T^{(k)}$ to an operator which is unitarily equivalent
to $S_{\omega}^{(k)}$ for all $w\in \varOmega_k$ and
$k=1,2$.

Assume that $T^{(1)}\cong T^{(2)}$. First, we note
that there is no loss of generality in assuming that
$\varOmega_1 = \varOmega_2=:\varOmega$ because, due to
$(T^{(1)})^*\cong (T^{(2)})^*$, we have
   \begin{align*}
\card{\varOmega_1}&=\dim \bigg(\bigoplus_{\omega\in
\varOmega_1} \ker \big(S_{\omega}^{(1)}\big)^* \bigg)
= \dim \ker (T^{(1)})^*
   \\
& = \dim \ker (T^{(2)})^* = \card{\varOmega_2}.
   \end{align*}
In turn, by \cite[Corollary~ 1]{Shi}, we can assume
that $\lambda_{\omega, n}^{(k)} > 0$ for all $\omega
\in \varOmega,$ $n\in \zbb_+$ and $k=1,2$. For
$k=1,2$, we denote by $\mathcal M^{(k)}$ the
orthogonal sum $\bigoplus_{\omega \in \varOmega}
\langle e_{\omega, 0}^{(k)} \rangle$ and by $W^{(k)}$
the operator valued unilateral weighted shift on
$\ell^2_{\mathcal M^{(k)}}$ with weights
$\{W_n^{(k)}\}_{n=0}^{\infty} \subseteq \B(\mathcal
M^{(k)})$ uniquely determined by the following
equations
   \begin{align*}
W_n^{(k)} e_{\omega, 0}^{(k)} = \lambda_{\omega,
n}^{(k)} e_{\omega, 0}^{(k)}, \quad \omega \in
\varOmega, \, n \in \zbb_+, \, k=1,2.
   \end{align*}
($W^{(k)}$ is well-defined because $\|T^{(k)}\| =
\sup_{n\in \zbb_+} \sup_{\omega\in \varOmega}
\lambda_{\omega, n}^{(k)} = \sup_{n\in
\zbb_+}\|W_n^{(k)}\|$.) We claim that $T^{(k)} \cong
W^{(k)}$ for $k=1,2$. Indeed, for $k=1,2,$ there
exists a unique unitary isomorphism $V_k\in
\B(\hh^{(k)},\ell^2_{\mathcal M^{(k)}})$ such that
   \begin{align*}
V_k e_{\omega, n}^{(k)} =
(\underset{\substack{\phantom{a} \\ \langle 0
\rangle}} 0, \ldots, 0, \underset{\langle n
\rangle}{e_{\omega, 0}^{(k)}}, 0, \dots), \quad \omega
\in \varOmega,\, n\in \zbb_+.
   \end{align*}
It is a matter of routine to show that $V_k T^{(k)}
e_{\omega, n}^{(k)}= W^{(k)} V_k e_{\omega, n}^{(k)}$
for all $\omega \in \varOmega,$ $n\in \zbb_+$ and
$k=1,2.$ This implies the claimed unitary equivalence.
As a consequence, we see that $W^{(1)} \cong W^{(2)}$.
Hence, by Corollary~ \ref{dom2}, there exists a
unitary isomorphism $U_0\in \B(\mathcal M^{(1)},
\mathcal M^{(2)})$ such that
   \begin{align} \label{Jur3}
U_0 W_n^{(1)} = W_n^{(2)} U_0, \quad n\in \zbb_+.
   \end{align}
Given $k,l \in \{1,2\}$ and $\omega_0 \in \varOmega$,
we set
   \begin{align*}
\varOmega_{\omega_0}^{(k,l)}=\big\{\omega\in \varOmega
\colon \lambda_{\omega, n}^{(k)} = \lambda_{\omega_0,
n}^{(l)} \, \forall n \in \zbb_+\big\}=\big\{\omega
\in \varOmega \colon S_{\omega}^{(k)} =
S_{\omega_0}^{(l)}\big\}.
   \end{align*}
Our next goal is to show that
   \begin{align} \label{Jur1}
\card \varOmega_{\omega_0}^{(1,1)} = \card
\varOmega_{\omega_0}^{(2,1)}, \quad \omega_0 \in
\varOmega.
   \end{align}
For this, fix $\omega_0 \in \varOmega$. It follows
from the injectivity of $U_0$ that \allowdisplaybreaks
   \begin{align} \notag
U_0 \bigg(\bigcap_{n=0}^{\infty} \ker
(\lambda_{\omega_0, n}^{(1)} I - W_n^{(1)})\bigg) & =
\bigcap_{n=0}^{\infty} U_0 \Big (\ker
(\lambda_{\omega_0, n}^{(1)} I - W_n^{(1)})\Big )
   \\ \label{Ko}
& \hspace{-1.7ex}\overset{\eqref{Jur3}}=
\bigcap_{n=0}^{\infty} \ker (\lambda_{\omega_0,
n}^{(1)} I - W_n^{(2)}).
   \end{align}
Since
   \begin{align*}
\ker (\lambda_{\omega_0, n}^{(1)} I - W_n^{(k)}) =
\bigoplus_{\substack{\omega\in \varOmega\colon \\
\lambda_{\omega, n}^{(k)} = \lambda_{\omega_0,
n}^{(1)}}} \langle e_{\omega, 0}^{(k)} \rangle, \quad
n \in \zbb_+, \, k=1,2,
   \end{align*}
and consequently
   \begin{align*}
\bigcap_{n=0}^{\infty} \ker (\lambda_{\omega_0,
n}^{(1)} I - W_n^{(k)}) = \bigoplus_{\omega\in
\varOmega_{\omega_0}^{(k,1)}} \langle e_{\omega,
0}^{(k)} \rangle, \quad k=1,2,
   \end{align*}
we deduce that \allowdisplaybreaks
   \begin{align*}
\card \varOmega_{\omega_0}^{(1,1)} &= \dim
\bigoplus_{\omega\in \varOmega_{\omega_0}^{(1,1)}}
\langle e_{\omega, 0}^{(1)} \rangle = \dim
\bigcap_{n=0}^{\infty} \ker (\lambda_{\omega_0,
n}^{(1)} I - W_n^{(1)})
   \\
& \hspace{-1.7ex}\overset{\eqref{Ko}}= \dim
\bigcap_{n=0}^{\infty} \ker (\lambda_{\omega_0,
n}^{(1)} I - W_n^{(2)}) = \card
\varOmega_{\omega_0}^{(2,1)}.
   \end{align*}
Hence, the condition \eqref{Jur1} holds. Since by
\eqref{Jur3}, $U_0^* W_n^{(2)} = W_n^{(1)} U_0^*$ for
all $n\in \zbb_+$, we infer from \eqref{Jur1} that
   \begin{align}  \label{Jur2}
\card \varOmega_{\omega_0}^{(2,2)} = \card
\varOmega_{\omega_0}^{(1,2)}, \quad \omega_0 \in
\varOmega.
   \end{align}
Using the equivalence relations $\mathcal R_k
\subseteq \varOmega \times \varOmega$, $k=1,2$,
defined by
   \begin{align*}
\omega \mathcal R_k \omega^{\prime} \iff
S_{\omega}^{(k)} = S_{\omega^{\prime}}^{(k)}, \quad
\omega, \omega^{\prime} \in \varOmega, \, k,l \in
\{1,2\},
   \end{align*}
and combining \eqref{Jur1} with \eqref{Jur2} we obtain
(ii).

(ii)$\Rightarrow$(i) This implication is obvious.
   \end{proof}
   \section{\label{Sec5}Unitary equivalence of $2$-isometries satisfying
the kernel condition} In view of the well-known
characterizations of the unitary equivalence of normal
operators (see e.g., \cite[Chap.\ 7]{b-s}), Lemma~
\ref{unrown} reduces the question of unitary
equivalence of $2$-isometries satisfying the kernel
condition to the consideration of pure operators in
this class. By Theorem~ \ref{Zak1} below, a
$2$-isometry satisfying the kernel condition is pure
if and only if it is unitarily equivalent to an
operator valued unilateral weighted shift $W$ on
$\ell^2_{\mathcal M}$ with weights
$\{W_n\}_{n=0}^{\infty}$ defined by \eqref{wagi}. Our
first goal is to give necessary and sufficient
conditions for two such operators to be unitarily
equivalent (see Theorem~ \ref{fincyc}). Next, we
discuss the question of when a pure $2$-isometry
satisfying the kernel condition is unitarily
equivalent to an orthogonal sum of unilateral weighted
shifts (see Theorem~ \ref{fincyc2}). This enables us
to answer the question of whether all finitely
multicyclic pure $2$-isometries satisfying the kernel
condition are necessarily finite orthogonal sums of
weighted shifts (see Corollary~ \ref{mulcyc}).

Before stating a model theorem for pure $2$-isometries
satisfying the kernel condition, we list some basic
properties of the sequence $\{\xi_n\}_{n=0}^{\infty}$
of self-maps of the interval $[1,\infty)$ which are
defined by
   \begin{align} \label{xin}
\xi_n(x) = \sqrt{\frac{1+ (n+1)(x^2-1)}{1+ n(x^2-1)}},
\quad x \in [1,\infty), \, n\in \zbb_+.
   \end{align}
   \begin{lemma} \label{xin11}
   \mbox{\phantom{a}}
\begin{enumerate}
   \item[(i)] $\xi_0$ is the
identity map,
   \item[(ii)] $\xi_{m+n} = \xi_m \circ \xi_n$
for all $m,n\in \zbb_+$,
   \item[(iii)] $\xi_n(1)=1$ for all $n\in \zbb_+$,
   \item[(iv)] $\xi_{n}(x) > \xi_{n+1}(x) > 1$ for all
$x \in (1,\infty)$ and $n\in \zbb_+,$
   \item[(v)] if $\{\zeta_n\}_{n=0}^{\infty}$ is a sequence of
self-maps of $[1,\infty)$ such that $\zeta_0$ is the
identity map and $\zeta_{n+1} =
\sqrt{\frac{2\zeta_{n}^2-1}{\zeta_{n}^2}}$ for all
$n\in \zbb_+,$ then $\zeta_{n}= \xi_{n}$ for all $n\in
\zbb_+.$
   \end{enumerate}
   \end{lemma}
The following model theorem, which is a part of
\cite[Theorem~ 2.5]{A-C-J-S}, classifies (up to
unitary equivalence) pure $2$-isometries satisfying
the kernel condition.
   \begin{theorem} \label{Zak1}
If $\hh \neq \{0\}$ and $T\in \B(\hh)$, then the
following are equivalent{\em :}
   \begin{enumerate}
   \item[(i)] $T$ is an analytic $2$-isometry satisfying the kernel
condition,
   \item[(ii)] $T$ is a completely non-unitary
$2$-isometry satisfying the kernel condition,
   \item[(iii)] $T$ is a pure $2$-isometry satisfying the kernel
condition,
   \item[(iv)] $T$ is unitarily equivalent to
an operator valued unilateral weighted shift $W$ on
$\ell^2_{\mathcal M}$ with
weights\footnote{\;\label{Foot2}Note that the sequence
$\{W_n\}_{n=0}^{\infty} \subseteq \B(\mathcal M)$
defined by \eqref{wagi} is uniformly bounded, and
consequently $W\in \B(\ell^2_{\mathcal M})$.}
$\{W_n\}_{n=0}^{\infty}$ given by
   \begin{align} \label{wagi}
   \left.
   \begin{gathered} W_n = \int_{[1,\infty)}
\xi_n(x) E(d x), \quad n \in \zbb_+,
   \\
   \begin{minipage}{63ex}
where $E$ is a compactly supported $\B(\mathcal
M)$-valued Borel spectral measure on the interval
$[1,\infty)$.
   \end{minipage}
   \end{gathered}
   \; \right\}
   \end{align}
   \end{enumerate}
   \end{theorem}
Now we answer the question of when two pure
$2$-isometries satisfying the kernel condition are
unitarily equivalent. We refer the reader to
\cite[Section~ 2.2]{JJS} (resp., \cite[Chapter~
7]{b-s}) for necessary information on the diagonal
operators (resp., the spectral type and the
multiplicity function of a selfadjoint operator, which
is a complete system of its unitary invariants).
   \begin{theorem} \label{fincyc}
Suppose $W\in \B(\ell^2_{\mathcal M})$ is an operator
valued unilateral weighted shift with weights
$\{W_n\}_{n=0}^{\infty}$ given by
   \begin{align*}
W_n = \int_{[1,\infty)} \xi_n(x) E(d x), \quad n \in
\zbb_+,
   \end{align*}
where $\{\xi_n\}_{n=0}^{\infty}$ are as in \eqref{xin}
and $E$ is a compactly supported $\B({\mathcal
M})$-valued Borel spectral measure on $[1,\infty)$.
Let $({\widetilde W}, {\widetilde{\mathcal M}},
\{\widetilde W_n\}_{n=0}^{\infty}, {\widetilde E})$ be
another such system. Then the following conditions are
equivalent{\em :}
   \begin{enumerate}
   \item[(i)] $W\cong \widetilde W$,
   \item[(ii)] $W_0\cong \widetilde W_0$,
   \item[(iii)] the spectral types and the multiplicity functions
of $W_0$ and $\widetilde W_0$ coincide,
   \item[(iv)] the spectral measures $E$ and
$\widetilde E$ are unitarily equivalent.
   \end{enumerate}
Moreover, if the operators $W_0$ and $\widetilde W_0$
are diagonal, then {\em (ii)} holds if and only if
   \begin{enumerate}
   \item[(v)] $\dim \ker(\lambda I - W_0) = \dim \ker(\lambda I -
\widetilde W_0)$ for all $\lambda \in \C$.
   \end{enumerate}
   \end{theorem}
   \begin{proof}
Since $\xi_0(x) = x$ for all $x \in [1,\infty)$, $E$
and $\widetilde E$ are the spectral measures of $W_0$
and $\widetilde W_0$, respectively. Hence, the
conditions (ii) and (iv) are equivalent. That (ii) and
(iii) are equivalent follows from \cite[Theorem~
7.5.2]{b-s}. Note that $\{W_n\}_{n=0}^{\infty}$ are
commuting positive bounded operators such that $W_n
\Ge I$ for all $n\in \zbb_+$. The same is true for
$\{\widetilde W_n\}_{n=0}^{\infty}$. Therefore, $W$
and $\widetilde W$ satisfy the assumptions of Theorem~
\ref{Fran4}.

(i)$\Rightarrow$(ii) This is a direct consequence of
Theorem~ \ref{Fran4}.

(iii)$\Rightarrow$(i) If $UE=\widetilde E U$, where
$U\in \B({\mathcal M}, \widetilde {\mathcal M})$ is a
unitary isomorphism, then by \cite[Theorem~
5.4.9]{b-s} $UW_n=\widetilde W_nU$ for $n \in \zbb_+$.
Hence, by Theorem~ \ref{Fran4}, $W\cong \widetilde W$.

It is a simple matter to show that if the operators
$W_0$ and $\widetilde W_0$ are diagonal, then the
conditions (ii) and (v) are equivalent. This completes
the proof.
   \end{proof}
It follows from Theorems~ \ref{Zak1} and \ref{fincyc}
that the spectral type and the multiplicity function
of the spectral measure of $W_0$ form a complete
system of unitary invariants for completely
non-unitary $2$-isometries satisfying the kernel
condition.

Theorem~ \ref{fincyc2} below answers the question of
when a completely non-unitary $2$-isometry satisfying
the kernel condition is unitarily equivalent to an
orthogonal sum of unilateral weighted shifts. In the
case when $\ell^2_{\mathcal M}$ is a separable Hilbert
space, this result can in fact be deduced from
\cite[Theorem~ 3.9]{lam-1}. There are two reasons why
we have decided to include the proof of Theorem~
\ref{fincyc2}. First, our result is stated for Hilbert
spaces which are not assumed to be separable. Second,
an essential part of the proof of Theorem~
\ref{fincyc2} will be used later in the proof of
Theorem~ \ref{2isscs-t}.
   \begin{theorem} \label{fincyc2}
Let $W\in \B(\ell^2_{\mathcal M})$ be an operator
valued unilateral weighted shift with weights
$\{W_n\}_{n=0}^{\infty}$ given by
   \begin{align} \label{wnint}
W_n = \int_{[1,\infty)} \xi_n(x) E(d x), \quad n \in
\zbb_+,
   \end{align}
where $\{\xi_n\}_{n=0}^{\infty}$ are as in \eqref{xin}
and $E$ is a compactly supported $\B({\mathcal
M})$-valued Borel spectral measure on $[1,\infty)$.
Then the following conditions are equivalent{\em :}
   \begin{enumerate}
   \item[(i)] $W \cong \bigoplus_{j\in J} S_j$, where $S_j$
are unilateral weighted shifts,
   \item[(ii)] $W_0$ is a diagonal operator.
   \end{enumerate}
Moreover, if {\em (i)} holds, then the index set $J$
is of cardinality $\dim \ker W^*.$
   \end{theorem}
   \begin{proof}
(ii)$\Rightarrow$(i) Since $W_0$ is a diagonal
operator and $W_0\Ge I$, there exists an orthonormal
basis $\{e_j\}_{j\in J}$ of $\mathcal M$ and a system
$\{\lambda_j\}_{j\in J} \subseteq [1,\infty)$ such
that
   \begin{align*}
W_0e_j = \lambda_j e_j, \quad j\in J.
   \end{align*}
By \eqref{aopws}, $\dim \ker W^*=\dim \mathcal
M=\text{the cardinality of $J$}$. Note that $E$, which
is the spectral measure of $W_0$, is given by
   \begin{align} \label{edel}
E(\varDelta) f = \sum_{j\in J}
\chi_{\varDelta}(\lambda_j) \langle f, e_j \rangle
e_j, \quad f \in \mathcal M, \, \varDelta \in
\borel{[1,\infty)}.
   \end{align}
Let $S_j$ be the unilateral weighted shift in $\ell^2$
with weights $\{\xi_n(\lambda_j)\}_{n=0}^{\infty}$. By
\cite[Lemma~ 6.1 and Proposition~ 6.2]{Ja-St}, $S_j$
is a $2$-isometry such that $\|S_j\|=\lambda_j$ for
every $j\in J$. Since $\sup_{j\in J} \lambda_j <
\infty$, we see that $\bigoplus_{j\in J} S_j \in
\B\big((\ell^2)^{\oplus{\mathfrak n}}\big)$, where
$\mathfrak{n}$ is the cardinal number of $J$. Define
the operator $V \colon \ell^2_{\mathcal M} \to
(\ell^2)^{\oplus{\mathfrak n}}$ by
   \begin{align*}
(V(h_0,h_1,\dots))_j = (\langle h_0, e_j\rangle,
\langle h_1, e_j\rangle, \ldots), \quad j \in J, \,
(h_0,h_1,\ldots) \in \ell^2_{\mathcal M}.
   \end{align*}
Since for every $(h_0,h_1,\ldots) \in \ell^2_{\mathcal
M}$,
   \begin{align*}
\sum_{j\in J} \sum_{n=0}^{\infty} |\langle h_n,
e_j\rangle|^2 = \sum_{n=0}^{\infty} \sum_{j\in J}
|\langle h_n, e_j\rangle|^2 = \sum_{n=0}^{\infty}
\|h_n\|^2 = \|(h_0,h_1,\ldots)\|^2,
   \end{align*}
the operator $V$ is an isometry. Note that for all
$j,k\in J$ and $m\in \zbb_+$,
   \begin{align*}
(V(\underset{\langle 0 \rangle} 0, \ldots, 0,
\underset{\langle m \rangle}{e_k}, 0, \ldots))_j =
   \begin{cases}
(0,0,\ldots) & \text{if } j\neq k,
   \\[1.5ex]
(\underset{\langle 0 \rangle} 0, \ldots, 0,
\underset{\langle m \rangle} 1, 0, \dots) & \text{if }
j=k,
   \end{cases}
   \end{align*}
which means that the range of $V$ is dense in
$(\ell^2)^{\oplus{\mathfrak n}}$. Thus $V$ is a
unitary isomorphism. It follows from \eqref{wnint}
that
   \begin{align} \label{wnej}
W_n e_j = \int_{[1,\infty)} \xi_n(x) E(dx) e_j
\overset{\eqref{edel}}= \xi_n(\lambda_j) e_j, \quad j
\in J, \, n \in \zbb_+.
   \end{align}
This implies that \allowdisplaybreaks
   \begin{align*}
VW(h_0,h_1,\ldots) & = \{(0, \langle W_0h_0,
e_j\rangle, \langle W_1h_1, e_j\rangle,
\ldots)\}_{j\in J}
   \\
& \hspace{-.7ex}\overset{\eqref{wnej}}= \{(0,
\xi_0(\lambda_j)\langle h_0, e_j\rangle,
\xi_1(\lambda_j)\langle h_1, e_j\rangle,
\ldots)\}_{j\in J}
   \\
& = \{S_j(V(h_0,h_1, \ldots))_j\}_{j\in J}
   \\
& = \Big(\bigoplus_{j\in J} S_j\Big)
V(h_0,h_1,\ldots), \quad (h_0,h_1,\ldots) \in
\ell^2_{\mathcal M}.
   \end{align*}

(i)$\Rightarrow$(ii) Suppose that $W\cong
\bigoplus_{j\in J} S_j$, where $S_j$ are unilateral
weighted shifts. Since $W$ is a $2$-isometry, so is
$S_j$ for every $j\in J$. Hence $S_j$ is injective for
every $j\in J$. As a consequence, there is no loss of
generality in assuming that the weights of $S_j$ are
positive (see \cite[Corollary~ 1]{Shi}). By
\cite[Lemma~ 6.1(ii)]{Ja-St}, for every $j\in J$ there
exists $\lambda_j \in [1,\infty)$ such that
$\{\xi_n(\lambda_j)\}_{n=0}^{\infty}$ are weights of
$S_j$. Let $\widetilde{\mathcal M}$ be a Hilbert space
such that $\dim \widetilde{\mathcal M}=\text{the
cardinality of $J$}$, $\{\tilde e_j\}_{j\in J}$ be an
orthonormal basis of $\widetilde{\mathcal M}$ and
$\widetilde E$ be a $\B(\widetilde{\mathcal
M})$-valued Borel spectral measure on $[1,\infty)$
given by
   \begin{align*}
\widetilde E(\varDelta) f = \sum_{j\in J}
\chi_{\varDelta}(\lambda_j) \langle f, \tilde e_j
\rangle \tilde e_j, \quad f \in \widetilde {\mathcal
M}, \, \varDelta \in \borel{[1,\infty)}.
   \end{align*}
Since by \cite[Proposition~ 6.2]{Ja-St}, $\sup_{j\in
J} \lambda_j = \sup_{j\in J} \|S_j\| < \infty$, the
spectral measure $\widetilde E$ is compactly supported
in $[1,\infty)$. Define the sequence $\{\widetilde
W_n\}_{n=0}^{{\infty}} \subseteq
\B(\widetilde{\mathcal M})$ by
   \begin{align*}
\widetilde W_n = \int_{[1,\infty)} \xi_n(x) \widetilde
E(d x), \quad n \Ge 0.
   \end{align*}
Note that the sequence $\{\widetilde
W_n\}_{n=0}^{{\infty}}$ is uniformly bounded (see
Footnote \ref{Foot2}). Clearly, $\widetilde W_0 \tilde
e_j = \lambda_j \tilde e_j$ for all $j \in J$, which
means that $\widetilde W_0$ is a diagonal operator.
Denote by $\widetilde W$ the operator valued
unilateral weighted shift on
$\ell^2_{\widetilde{\mathcal M}}$ with weights
$\{\widetilde W_n\}_{n=0}^{{\infty}}$. It follows from
the proof of the implication (ii)$\Rightarrow$(i) that
$\widetilde W \cong \bigoplus_{j\in J} S_j$. Hence $W
\cong \widetilde W$. By Theorem~ \ref{fincyc}, $W_0$
is a diagonal operator.
   \end{proof}
   \begin{remark}
Regarding Theorem~ \ref{fincyc2}, it is worth noting
that if $\dim \ker W^* \Le \aleph_0$ and $W_0$ is
diagonal, then $W$ can be modelled by a weighted
composition operator on an $L^2$-space over a
$\sigma$-finite measure space (use \cite[Section
2.3(g)]{BJJS2} and an appropriately adapted version of
\cite[Corollary~ C.2]{BJJS1}). \hfill $\diamondsuit$
   \end{remark}
Recall that for a given operator $T \in B (\hh)$, the
smallest cardinal number $\mathfrak n$ for which there
exists a closed vector subspace $\mathcal N$ of $\hh$
such that $\dim {\mathcal N} = \mathfrak n$ and $\hh =
\bigvee_{n=0}^{\infty} T^n(\mathcal N)$ is called the
{\em order of multicyclicity} of $T$. If the order of
multicyclicity of $T$ is finite, then $T$ is called
{\em finitely multicyclic}. As shown in Lemma~
\ref{mulcyc1} below, the order of multicyclicity of a
completely non-unitary $2$-isometry can be calculated
explicitly (in fact, the proof of Lemma~ \ref{mulcyc1}
contains more information). Part (i) of Lemma~
\ref{mulcyc1} appeared in \cite[Proposition~
1(i)]{Her} with a slightly different definition of the
order of multicyclicity and a different proof. Part
(ii) of Lemma~ \ref{mulcyc1} is covered by
\cite[Lemma~ 2.19(b)]{Ch-0} in the case of finite
multicyclicity. In fact, the proof of part (ii) of
Lemma~ \ref{mulcyc1}, which is given below, works for
analytic operators having Wold-type decomposition in
the sense of Shimorin (see \cite{Sh}).
   \begin{lemma} \label{mulcyc1}
Let $T\in \B(\hh)$ be an operator. Then
   \begin{enumerate}
   \item[(i)] the order of
multicyclicity of $T$ is greater than or equal to
$\dim \ker T^*,$
   \item[(ii)] if $T$ is a completely non-unitary
$2$-isometry, then the order of multicyclicity of $T$
is equal to $\dim \ker T^*.$
   \end{enumerate}
   \end{lemma}
   \begin{proof}
(i) Let $\mathcal N$ be a closed vector subspace of
$\hh$ such that $\hh = \bigvee_{n=0}^{\infty}
T^n(\mathcal N)$ and $P\in \B(\hh)$ be the orthogonal
projection of $\hh$ onto $\ker T^*.$ Clearly, $\ker
T^* \perp T^n(\hh)$ for all $n \in\nbb.$ If $f \in
\ker T^* \ominus \overline{P(\mathcal N)}$, then
   \begin{align*}
\langle f, T^0 h \rangle = \langle f, P h \rangle = 0,
\quad h \in {\mathcal N},
   \end{align*}
which together with the previous statement yields $f
\in \big(\bigvee_{n=0}^{\infty} T^n(\mathcal N)
\big)^{\perp} = \{0\}.$ Hence $\overline{P(\mathcal
N)} = \ker T^*.$ As a consequence, the operator
$P|_{\mathcal N}\colon \mathcal N\to \ker T^*$ has
dense range, which implies that $\dim \ker T^* \Le
\dim \mathcal N$ (see \cite[Problem 56]{Hal}). This
gives~ (i).

(ii) Since, by \cite[Theorem~ 3.6]{Sh}, $\hh =
\bigvee_{n=0}^{\infty} T^n(\ker T^*),$ we see that the
order of multicyclicity of $T$ is less than or equal
to $\dim \ker T^*.$ This combined with (i) completes
the proof.
   \end{proof}
The following result generalizes the remarkable fact
that a finitely multicyclic completely non-unitary
isometry is unitarily equivalent to an orthogonal sum
of finitely many unilateral unweighted shifts (cf.
\cite[Proposition~ 2.4]{Kub}).
   \begin{corollary} \label{mulcyc}
A finitely multicyclic completely non-unitary
$2$-isometry $T$ satisfying the kernel condition is
unitarily equivalent to an orthogonal sum of
$\mathfrak n$ unilateral weighted shifts, where
$\mathfrak n$ equals the order of multicyclicity of
$T$. Moreover, for each cardinal number $\mathfrak n
\Ge \aleph_0$ there exists a completely non-unitary
$2$-isometry satisfying the kernel condition, whose
order of multicyclicity equals $\mathfrak n$ and which
is not unitarily equivalent to any orthogonal sum of
unilateral weighted shifts.
   \end{corollary}
   \begin{proof}
Apply Theorem~ \ref{fincyc2}, Lemma~ \ref{mulcyc1} and
the fact that positive operators in finite-dimensional
Hilbert spaces are diagonal while in
infinite-dimensional not necessarily.
   \end{proof}
   \section{\label{Sec9} Unitary equivalence of $2$-isometric
weighted shifts on directed trees satisfying the
kernel condition} This section provides a model for a
$2$-isometric weighted shift $\slam$ on a rooted
directed tree $\tcal$ which satisfy the condition
\eqref{hypo+} (see Theorem~ \ref{2isscs-t}). Although
the kernel condition is weaker than the condition
\eqref{hypo+}, both are equivalent if $\tcal$ is
leafless and the weights of $\slam$ are nonzero. The
aforesaid model enables us to classify (up to unitary
equivalence) $2$-isometric weighted shifts on rooted
directed trees satisfying the condition \eqref{hypo+}
in terms of $k$th generation branching degree (see
Theorem~ \ref{equival}).

We begin with necessary information on weighted shifts
on directed trees. The reader is referred to
\cite{JJS} for more details on this subject (see also
\cite{B-D-P,K-L-P,M-A} for very recent developments).
Let $\tcal = (V,E)$ be a directed tree (if not stated
otherwise, $V$ and $E$ stand for the sets of vertices
and edges of $\tcal$, respectively). If $\tcal$ has a
root, we denote it by $\rot$. We set $V^{\circ}=V$ if
$\tcal$ is rootless and $V^{\circ}=V\setminus
\{\rot\}$ otherwise. We say that $\tcal$ is {\em
leafless} if $V=V^{\prime}$, where $V^{\prime} := \{u
\in V\colon \child{u} \neq \emptyset\}$. Given
$W\subseteq V$ and $n\in \zbb_+,$ we set
$\childn{n}{W}=W$ if $n=0$ and
$\childn{n}{W}=\child{\childn{n-1}{W}}$ if $n\Ge 1$,
where $\child{W} = \bigcup_{u\in W} \{v\in V \colon
(u,v) \in E\}.$ We put $\des{W}=\bigcup_{n=0}^{\infty}
\childn{n}{W}.$ Given $v\in V$, we write
$\child{v}=\child{\{v\}}$ and
$\childn{n}{v}=\childn{n}{\{v\}}$. For $v\in
V^{\circ}$, a unique $u\in V$ such that $(u,v)\in E$
is said to be the {\em parent} of $v$; we denote it by
$\parent{v}$. The cardinality of $\child{v}$ is called
the {\em degree} of a vertex $v\in V$ and denoted by
$\deg{v}$. Recall that if $\tcal$ is rooted, then by
\cite[Corollary~ 2.1.5]{JJS}, we have
   \begin{align} \label{ind1}
V = \bigsqcup_{n=0}^\infty \childn{n} {\rot} \quad
\text{(the disjoint sum)}.
   \end{align}

   Following \cite[page 67]{JJS}, we define the
   directed tree $\tcal_{\eta,\kappa} =
   (V_{\eta,\kappa}, E_{\eta,\kappa})$ by
   \allowdisplaybreaks
   \begin{align}  \label{tetak}
   \left.
   \begin{aligned} V_{\eta,\kappa} & = \big\{-k\colon k\in
J_\kappa\big\} \cup \{0\} \cup \big\{(i,j)\colon i\in
J_\eta,\, j \in J_{\infty}\big\},
   \\
E_{\eta,\kappa} & = E_\kappa \cup
\big\{(0,(i,1))\colon i \in J_\eta\big\} \cup
\big\{((i,j),(i,j+1))\colon i\in J_\eta,\, j\in
J_{\infty}\big\},
   \\
E_\kappa & = \big\{(-k,-k+1) \colon k\in
J_\kappa\big\},
   \end{aligned}
   \;\right\}
   \end{align}
where $\eta \in \{2,3,4,\ldots\} \cup \{\infty\}$,
$\kappa \in \zbb_+ \cup \{\infty\}$ and $J_\iota = \{k
\in \zbb_+\colon 1 \Le k\Le \iota\}$ for $\iota \in
\zbb_+ \sqcup \{\infty\}$. The directed tree
$\tcal_{\eta,\kappa}$ is leafless, it has only one
branching vertex $0$ and $\deg{0}=\eta$. Moreover, it
is rooted if $\kappa < \infty$ and rootless if
$\kappa=\infty$.

Let $\tcal=(V,E)$ be a directed tree. In what follows
$\ell^2(V)$ stands for the Hilbert space of square
summable complex functions on $V$ equipped with the
standard inner product. If $W$ is a nonempty subset of
$V,$ then we regard the Hilbert space $\ell^2(W)$ as a
closed vector subspace of $\ell^2(V)$ by identifying
each $f\in \ell^2(W)$ with the function $\widetilde f
\in \ell^2(V)$ which extends $f$ and vanishes on the
set $V \setminus W$. Note that the set $\{e_u\}_{u\in
V}$, where $e_u\in \ell^2(V)$ is the characteristic
function of $\{u\}$, is an orthonormal basis of
$\ell^2(V)$. Given a system $\lambdab =
\{\lambda_v\}_{v\in V^{\circ}}$ of complex numbers, we
define the operator $\slam$ in $\ell^2(V)$, called a
{\em weighted shift on} $\tcal$ with weights
$\lambdab$ (or simply a weighted shift on $\tcal$), as
follows
   \begin{align*}
   \begin{aligned}
\mathscr D(\slam) & = \{f \in \ell^2(V) \colon
\varLambda_\tcal f \in \ell^2(V)\},
   \\
\slam f & = \varLambda_\tcal f, \quad f \in \mathscr
D(\slam),
   \end{aligned}
   \end{align*}
where $\mathscr D(\slam)$ stands for the {\em domain}
of $\slam$ and $\varLambda_\tcal$ is the mapping
defined on complex functions $f$ on $V$ by
   \begin{align*}
(\varLambda_\tcal f) (v) =
   \begin{cases}
\lambda_v \cdot f\big(\parent v\big) & \text{if } v\in
V^\circ,
   \\
   0 & \text{if } v \text{ is a root of } \tcal.
   \end{cases}
   \end{align*}

Now we collect some properties of weighted shifts on
directed trees that are needed in this paper (see
\cite[Propositions~ 3.1.3, 3.1.8, 3.4.3 and
3.5.1]{JJS}). From now on, we adopt the convention
that $\sum_{v\in \emptyset} x_v = 0$.
   \begin{lemma} \label{basicws}
Let $\slam$ be a weighted shift on $\tcal$ with
weights $\lambdab=\{\lambda_v\}_{v\in V^{\circ}}$.
Then
   \begin{enumerate}
   \item[(i)] $e_u$ is in $\mathcal D(\slam)$ if and
only if $\sum_{v \in \child{u}}|\lambda_v|^2 <
\infty;$ if $e_u \in \mathscr D(\slam)$, then $\slam
e_u = \sum_{v \in \child{u}}\lambda_v e_v$ and
$\|\slam e_u\|^2 = \sum_{v \in
\child{u}}|\lambda_v|^2,$
   \item[(ii)]
$\slam \in \B(\ell^2(V))$ if and only if $\sup_{u\in
V} \sum_{v \in \child{u}}|\lambda_v|^2 < \infty;$ if
this is the case, then $\|\slam\|^2=\sup_{u\in V}
\|\slam e_u\|^2 = \sup_{u\in V} \sum_{v \in
\child{u}}|\lambda_v|^2.$
   \end{enumerate}
Moreover, if $\slam \in \B(\ell^2(V))$, then
   \begin{enumerate}
   \item[(iii)] $\ker{\slam^*} =
   \begin{cases}
\langle e_{\rot} \rangle \oplus \bigoplus_{u \in
V^\prime} \big(\ell^2(\child{u}) \ominus \langle
\lambdab^u \rangle\big) & \text{if $\tcal$ is rooted,}
   \\[.5ex]
\bigoplus_{u \in V^\prime} \big(\ell^2(\child{u})
\ominus \langle \lambdab^u \rangle\big) &
\text{otherwise,}
   \end{cases}$
\\[1ex]
where $\lambdab^u \in \ell^2(\child{u})$ is given by
$\lambdab^u\colon \child{u} \ni v \to \lambda_v \in
\C$,
   \item[(iv)]
$|\slam| e_u = \|\slam e_u\|e_u$ for all $u\in V,$
   \end{enumerate}
   \end{lemma}
According to \cite[Lemma~ 5.3(viii)]{A-C-J-S}, bounded
weighted shifts on rooted directed trees are
completely non-unitary. As shown in Example~
\ref{obustr} below, this is no longer true for bounded
weighted shifts on rootless directed trees even though
they are isometric and non-unitary (note that, by
\cite[Proposition~ 3.5]{Ja-St}, $2$-isometric
bilateral weighted shifts are always unitary).
   \begin{example} \label{obustr}
Let us consider any isometric weighted shift $\slam$
on the directed tree $\tcal_{\eta,\infty}$ (see
\eqref{tetak}) with weights
$\lambdab=\{\lambda_v\}_{v\in V_{\eta,\infty}}$, where
$\eta \in \{2,3,4, \ldots\} \cup \{\infty\}$ is fixed.
This means that $\sum_{i=1}^{\eta}
|\lambda_{i,1}|^2=1$ and $|\lambda_{i,j}| =
|\lambda_{-k}|=1$ for all $i\in J_{\eta}$, $j\in
J_{\infty} \setminus \{1\}$ and $k\in \zbb_+$. We will
show that $\slam$ is non-unitary and it is not
completely non-unitary. For this, by Wold's
decomposition theorem (see \cite[Theorem~ 1.1]{SF}),
it suffices to prove that $\ker \slam^*\neq \{0\}$ and
$\bigoplus_{n=0}^{\infty} \slam^n(\ker \slam^*) \neq
\ell^2(V_{\eta,\infty})$. In view of Lemma~
\ref{basicws}(iii), we have
   \begin{align} \label{orthd}
\ker \slam^* = \bigoplus_{v \in V_{\eta,\infty}}
\Big(\ell^2(\child{v}) \ominus \langle \lambdab^{v}
\rangle\Big).
   \end{align}
Since $\eta \Ge 2$ and $\lambdab^{v}\neq 0$ for all
$v\in V_{\eta,\infty}$, we deduce that the only
nonzero term in the orthogonal decomposition
\eqref{orthd} is $\ell^2(\child{0}) \ominus \langle
\lambdab^{0} \rangle$. Hence $\ker \slam^*\neq \{0\}$
and
   \begin{align*}
\bigoplus_{n=0}^{\infty} \slam^n(\ker \slam^*)
\subseteq \chi_{\varOmega} \cdot
\ell^2(V_{\eta,\infty}) \neq \ell^2(V_{\eta,\infty}),
   \end{align*}
where $\varOmega = \bigcup_{n=1}^{\infty}
\childn{n}{0}$. This proves our claim. \hfill
$\diamondsuit$
   \end{example}
   \begin{remark}
By \cite[Remark~ 5.8 and Proposition~ 5.11]{A-C-J-S},
a $2$-isometric weighted shift on a rootless directed
tree with nonzero weights which satisfies the kernel
condition is isometric. Further, if $\slam$ is an
isometric weighted shift on a rootless directed tree,
then by Wold's decomposition theorem, it is (up to
unitary equivalence) an orthogonal sum $W \oplus
S^{\oplus{\mathfrak n}}$, where $W$ is a unitary
operator, $S$ is the isometric unilateral shift of
multiplicity $1$ and ${\mathfrak n}=\dim \ker
\slam^*$. In particular, the isometry $\slam$ in
Example~ \ref{obustr} is equal to $U \oplus S^{\oplus
(\eta - 1)}$, where $U$ is the unitary bilateral shift
of multiplicity~ $1$.
   \hfill $\diamondsuit$
   \end{remark}
Recall that a weighted shift $\slam \in \B(\ell^2(V))$
on a leafless directed tree $\tcal$ with nonzero
weights $\lambdab=\{\lambda_v\}_{v\in V^{\circ}}$
satisfies the kernel condition if and only if there
exists a family $\{\alpha_v\}_{v\in V} \subseteq
\rbb_+$ such that
   \begin{align}  \label{hypo+}
\|\slam e_u\|=\alpha_{\parent{u}}, \quad u \in
V^{\circ}.
   \end{align}
In general, the condition \eqref{hypo+} is stronger
than the kernel condition (see \cite[Lemma~
5.6]{A-C-J-S}). In view of \cite[Remark~ 5.8 and
Proposition~ 5.10]{A-C-J-S}, if $\slam\in
\B(\ell^2(V))$ is a $2$-isometric weighted shift on a
rooted directed tree $\tcal$ with nonzero weights
$\lambdab=\{\lambda_v\}_{v\in V^{\circ}}$ which
satisfies the kernel condition, then $\tcal$ is
leafless, $\|\slam e_v\| = \mathrm{const}$ on
$\childn{n}{\rot}$ for every $n\in \zbb_+$, and the
corresponding sequence of constants forms a sequence
of positive weights of a $2$-isometric unilateral
weighted shift (cf.\ \cite[Lemma~ 6.1(ii)]{Ja-St}).
This suggests the following method of constructing
such $\slam$'s.
   \begin{procedure} \label{uwrem}
Let $\tcal$ be a rooted and leafless directed tree.
Take a sequence $\{\beta_n\}_{n=0}^{\infty}$ of
positive weights of a $2$-isometric unilateral
weighted shift. By \cite[Lemma~ 6.1(ii)]{Ja-St}, there
exists $x\in [1,\infty)$ such that $\beta_n=\xi_n(x)$
for all $n\in\zbb_+$ (the converse statement is true
as well). Then, using \eqref{ind1} and the following
equation (cf.\ \cite[Eq.\ (2.2.6)]{BJJS})
   \begin{align*}
\childn{n+1}{\rot} = \bigsqcup_{u\in \childn{n}{\rot}}
\child{u}, \quad n\in \zbb_+,
   \end{align*}
we can define inductively for every $n\in\zbb_+$ the
system $\{\lambda_v\}_{v\in \childn{n+1}{\rot}}$ of
complex numbers (not necessarily nonzero) such that
$\sum_{w\in \child{u}} |\lambda_w|^2 = \beta_{n}^2$
for all $u\in \childn{n}{\rot}$. Let $\slam$ be the
weighted shift on $\tcal$ with the so-constructed
weights $\lambdab=\{\lambda_v\}_{v \in V^{\circ}}.$
Clearly, in view of Lemma~ \ref{basicws}(i), we have
   \begin{align*}
x=\beta_0=\|\slam e_{\rot}\|.
   \end{align*}
Since the sequence $\{\xi_n(t)\}_{n=0}^{\infty}$ is
monotonically decreasing for every $t\in [1,\infty)$
(see Lemma~ \ref{xin11}), we infer from \eqref{ind1}
and Lemma~ \ref{basicws}(ii) that $\slam \in
\B(\ell^2(V))$ and $\beta_0=\|\slam\|.$ By
\cite[Proposition~ 5.10]{A-C-J-S}, $\slam$ is a
$2$-isometric weighted shift on $\tcal$ which
satisfies \eqref{hypo+} for some $\{\alpha_v\}_{v\in
V} \subseteq \rbb_+.$ Hence, according to \cite[Lemma~
5.6]{A-C-J-S}, $\slam$ satisfies the kernel condition.
\hfill $\diamondsuit$
   \end{procedure}
We will show in Theorem~ \ref{2isscs-t} below that a
$2$-isometric weighted shift on a rooted directed tree
which satisfies \eqref{hypo+} is unitarily equivalent
to an orthogonal sum of $2$-isometric unilateral
weighted shifts with positive weights; the orthogonal
sum always contains a ``basic'' $2$-isometric
unilateral weighted shift with weights
$\{\xi_n(x)\}_{n=0}^{\infty}$ for some $x\in
[1,\infty)$ and a number of inflations of
$2$-isometric unilateral weighted shifts with weights
$\{\xi_n(x)\}_{n=k}^{\infty}$, where $k$ varies over a
(possibly empty) subset of $\nbb$ (cf.\ Remark~
\ref{3uw}).

For $x \in [1,\infty)$, we denote by $S_{[x]}$ the
unilateral weighted shift in $\ell^2$ with weights
$\{\xi_n(x)\}_{n=0}^{\infty}$, where
$\{\xi_n\}_{n=0}^{\infty}$ is as in \eqref{xin}. Given
a leafless directed tree $\tcal$ and $k\in \nbb$, we
define the {\em $k$th generation branching degree}
${\mathfrak j}^{\tcal}_k$ of $\tcal$ by
   \begin{align} \label{defjn}
{\mathfrak j}^{\tcal}_k = \sum_{u \in
\childn{k-1}{\rot}} (\deg{u}-1), \quad k\in \nbb.
   \end{align}
Let us note that the proof of Theorem~
\ref{2isscs-t}(i) relies on the technique involved in
the proof of the implication (iii)$\Rightarrow$(v) of
\cite[Theorem~ 2.5]{A-C-J-S}.
   \begin{theorem} \label{2isscs-t}
  \mbox{\phantom{a}}
   \begin{enumerate}
   \item[(i)] Let $\slam \in \B(\ell^2(V))$ be a $2$-isometric
weighted shift on a rooted directed tree $\tcal$
satisfying \eqref{hypo+} for some $\{\alpha_v\}_{v\in
V} \subseteq \rbb_+.$ Then $\tcal$ is leafless and
$\slam$ is unitarily equivalent to the orthogonal sum
   \begin{align}  \label{zenob}
S_{[x]} \oplus \bigoplus_{k = 1}^{\infty}
\big(S_{[\xi_k(x)]}\big)^{\oplus {j}_k},
   \end{align}
where $x =\|\slam e_{\rot}\|$ and ${j}_k = {\mathfrak
j}^{\tcal}_k$ for all $k\in \nbb.$ Moreover, if the
weights of $\slam$ are nonzero, then ${j}_k \Le
\aleph_0$ for all $k\in \nbb$.
   \item[(ii)]
For any $x\in [1,\infty)$ and any sequence of cardinal
numbers $\{j_k\}_{k=1}^{\infty},$ the orthogonal sum
\eqref{zenob} is unitarily equivalent to a
$2$-isometric weighted shift $\slam \in \B(\ell^2(V))$
on a rooted directed tree $\tcal$ satisfying
\eqref{hypo+} for some $\{\alpha_v\}_{v\in V}
\subseteq \rbb_+$ such that $x=\|\slam e_{\rot}\|.$
Moreover, if $j_k \Le \aleph_0$ for all $k\in \nbb$,
then the weights of $\slam$ can be chosen to be~
positive.
   \end{enumerate}
   \end{theorem}
   \begin{proof}
(i) First, observe that by \cite[Lemma~ 5.7]{A-C-J-S},
$\tcal$ is leafless. To prove the unitary equivalence
part, we show that $\slam$ is unitarily equivalent to
an operator valued unilateral weighted shift
$\widetilde W$ on $\ell^2_{\ker{\slam^*}}$ with
weights $\{\widetilde W_n\}_{n=0}^{\infty}$ satisfying
the assumptions of Theorem~ \ref{fincyc2} and that
$\widetilde W_0$ is a diagonal operator.

It follows from \eqref{hypo+} and \cite[Lemma~
5.6]{A-C-J-S} that $T:=\slam$ satisfies the kernel
condition. By Lemma~ \ref{basicws}(iii), $\ker T^*\neq
\{0\}$ and so $T$ is a non-unitary $2$-isometry.
Hence, by \cite[Theorem~ 2.5]{A-C-J-S}, the spaces
$\{T^n (\ker T^*)\}_{n=0}^{\infty}$ are mutually
orthogonal. Since, by \cite[Lemma~5.3(viii)]{A-C-J-S},
$T$ is analytic, we infer from \cite[Theorem~ 3.6]{Sh}
that
   \begin{align} \label{ella}
\ell^2(V) = \bigoplus_{n=0}^{\infty} \mathcal M_n,
   \end{align}
where $\mathcal M_n:=T^n(\ker T^*)$ for $n\in \zbb_+$.
Given that $T$ is non-unitary and left-invertible, we
see that $\mathcal M_n$ is a nonzero closed vector
subspace of $\ell^2(V)$ and $\varLambda_n:=
T|_{\mathcal M_n}\colon \mathcal M_n \to \mathcal
M_{n+1}$ is a linear homeomorphism for every $n\in
\zbb_+$. Therefore, by \cite[Problem 56]{Hal}, the
Hilbert spaces $\mathcal M_n$ and $\mathcal M_0$ are
unitarily equivalent for every $n\in \zbb_+$. Set
$V_0=I_{\mathcal M_0}$. Let $\varLambda_0 = U_0
|\varLambda_0|$ be the polar decomposition of
$\varLambda_0$. Then $U_0\colon \mathcal M_0 \to
\mathcal M_1$ is a unitary isomorphism. Put
$V_1=U_0^{-1}\colon \mathcal M_1 \to \mathcal M_0$.
For $n\Ge 2$, let $V_n\colon \mathcal M_n \to \mathcal
M_0$ be any unitary isomorphism. By \eqref{ella}, we
can define the unitary isomorphism $V\colon \ell^2(V)
\to \ell^2_{\mathcal M_0}$ by
   \begin{align*}
V(h_0 \oplus h_1 \oplus \ldots) = (V_0h_0, V_1h_1,
\ldots), \quad h_0 \oplus h_1 \oplus \ldots \in
\ell^2(V).
   \end{align*}
Let $W\in \B(\ell^2_{\mathcal M_0})$ be the operator
valued unilateral weighted shift with (uniformly
bounded) invertible weights $\{V_{n+1} \varLambda_n
V_n^{-1}\}_{n=0}^{\infty} \subseteq \B(\mathcal M_0)$.
It is a routine matter to verify that $VT = WV.$
Therefore, $T=\slam$ is unitarily equivalent to $W$.
Since the zeroth weight of $W$, say $W_0$, equals $V_1
\varLambda_0 V_0^{-1},$ we get $W_0=|\varLambda_0|$. A
careful look at the proof of \cite[Proposition~
2.2]{Ja-3} reveals that $W$ is unitarily equivalent to
a $2$-isometric operator valued unilateral weighted
shift $\widetilde W$ on $\ell^2_{\mathcal M_0}$ with
invertible weights $\{\widetilde W_n\}_{n=0}^{\infty}
\subseteq \B(\mathcal M_0)$ such that $\widetilde W_0
= |W_0|$ and $\widetilde W_n \, \cdots\, \widetilde
W_0 \Ge 0$ for all $n \in \zbb_+$. Thus
   \begin{align} \label{w0l0}
\widetilde W_0=|\varLambda_0|.
   \end{align}
By \cite[Lemma~1]{R-0}, $\|\widetilde Wh\| \Ge \|h\|$
for all $h\in \ell^2_{\mathcal M_0}$, which yields
   \begin{align*}
\|\widetilde W_0 h_0\| = \|(0, \widetilde W_0 h_0, 0,
\ldots)\| = \|\widetilde W(h_0, 0, \ldots)\| \Ge
\|h_0\|, \quad h_0 \in \mathcal M_0.
   \end{align*}
Hence, by \eqref{w0l0}, $\widetilde W_0 \Ge I$. This
combined with the proof of \cite[Theorem~ 3.3]{Ja-3}
and Lemma~ \eqref{xin11}(v) implies that

   \begin{align*}
\widetilde W_n = \int_{[1,\|\widetilde W_0\|]}
\xi_n(x) E(d x), \quad n\in \zbb_+,
   \end{align*}
where $E$ is the spectral measure of $\widetilde W_0.$

Our next goal is to show that
   \begin{align} \label{w0more}
\text{${\mathcal M_0}$ reduces $|\slam|$ and
$\widetilde W_0 = |\slam||_{\mathcal M_0}$.}
   \end{align}
For this, observe that $\slam$ extends the operator
$\varLambda_0\colon \mathcal M_0 \to \mathcal M_1$ and
consequently
   \begin{align} \label{herb}
\langle \varLambda_0^*\varLambda_0 f,g\rangle =
\langle \slam^*\slam f,g\rangle, \quad f, g \in
\mathcal M_0.
   \end{align}
Knowing that $\slam$ satisfies the kernel condition,
we infer from \eqref{herb} that
$\varLambda_0^*\varLambda_0 = \slam^*\slam|_{\mathcal
M_0}$. This means that the orthogonal projection of
$\ell^2(V)$ onto $\mathcal M_0$ commutes with
$\slam^*\slam$. By the square root theorem, it
commutes with $|\slam|$ as well, which together with
\eqref{w0l0} implies \eqref{w0more}.

It follows from \eqref{ind1} and Lemma~
\ref{basicws}(iii) that
   \begin{align} \label{w0more-3}
\mathcal M_0 = \ker \slam^* = \langle e_{\rot} \rangle
\oplus \bigoplus_{k=1}^{\infty} \mathcal G_k,
   \end{align}
where $\mathcal G_{k}=\bigoplus_{u\in
\childn{k-1}{\rot}} \big(\ell^2(\child{u}) \ominus
\langle \lambdab^{u} \rangle\big)$ for $k\in \nbb$. In
view of Lemma~ \ref{basicws}(iv) and \eqref{hypo+}, we
see that $|\slam| e_{\rot} = \|\slam e_{\rot}\|
e_{\rot}$ and
   \begin{align*}
|\slam| f = \sum_{v\in \child{u}} f(v) |\slam| e_v =
\alpha_u f, \quad f \in \ell^2(\child{u}), \, u \in V.
   \end{align*}
This combined with \eqref{w0more} and \cite[Lemma~
5.9(iii)]{A-C-J-S} implies that
   \begin{align}  \label{afew}
   \left.
   \begin{aligned}
&\text{$\widetilde W_0$ is a diagonal operator,}
   \\
& \text{$\langle e_{\rot} \rangle$ reduces $\widetilde
W_0$ and $\widetilde W_0|_{\langle e_{\rot} \rangle} =
x I_{\langle e_{\rot} \rangle}$ with $x:=\|\slam
e_{\rot}\|$,}
   \\
& \text{$\mathcal G_k$ reduces $\widetilde W_0$ and
$\widetilde W_0|_{\mathcal G_k}=\xi_k(x) I_{\mathcal
G_k}$ for every $k\in \nbb$.}
   \end{aligned}
   \; \right\}
   \end{align}
Since $2$-isometries are injective and, by Lemma~
\ref{basicws}(i), $\|\slam e_u\|^2 = \sum_{v \in
\child{u}}|\lambda_v|^2$, we see that $\lambdab^{u}
\neq 0$ for every $u\in V$. As a consequence, we have
   \begin{align} \label{afew2}
\dim \mathcal G_k = \sum_{u \in \childn{k-1}{\rot}}
(\deg{u}-1)={\mathfrak j}^{\tcal}_k, \quad k\in\nbb.
   \end{align}
Now, following the proof of the implication
(ii)$\Rightarrow$(i) of Theorem~ \ref{fincyc2} and
applying \eqref{w0more-3}, \eqref{afew} and
\eqref{afew2} we see that $\slam$ is unitarily
equivalent to the orthogonal sum \eqref{zenob}. The
``moreover'' part is a direct consequence of
\cite[Proposition~ 3.1.10]{JJS}.

(ii) Let $\{j_k\}_{k=1}^{\infty}$ be a sequence of
cardinal numbers and $x\in [1,\infty).$ First, we
construct a directed tree $\tcal$. Without loss of
generality, we may assume that the set $\{n \in
\nbb\colon j_n \Ge 1\}$ is nonempty. Let $1 \Le n_1 <
n_2 < \ldots$ be a (finite or infinite) sequence of
positive integers such that
   \begin{align*}
\{n \in \nbb\colon j_n \Ge 1\} = \{n_1, n_2, \ldots\}.
   \end{align*}
Then using induction one can construct a leafless
directed tree $\tcal=(V,E)$ with root $\rot$ such that
each set $\childn{n_k-1}{\rot}$ has exactly one vertex
of degree $1+j_{n_k}$ and these particular vertices
are the only vertices in $V$ of degree greater than
one; clearly, the other vertices of $V$ are of degree
one (see Figure \ref{Fig1c-general}). Note that if $k
\Ge 3$, then a directed tree with these properties is
not unique (up to graph-isomorphism). Using Procedure~
\ref{uwrem}, we can find a system $\lambdab =
\{\lambda_v\}_{v\in V^{\circ}} \subseteq \rbb_+$ such
that $\slam \in \B(\ell^2(V))$, $\slam$ is a
$2$-isometry which satisfies \eqref{hypo+} for some
$\{\alpha_v\}_{v\in V} \subseteq \rbb_+$ and
$x=\|\slam e_{\rot}\|$. If additionally $j_n \Le
\aleph_0$ for all $n\in \nbb$, then the weights
$\{\lambda_v\}_{v\in V^{\circ}}$ can be chosen to be
positive (consult Procedure~ \ref{uwrem}). Since
   \begin{align*}
j_n=\sum_{u \in \childn{n-1}{\rot}} (\deg{u}-1), \quad
n\in \nbb,
   \end{align*}
we deduce from Theorem~ \ref{2isscs-t} that $T\cong
\slam$.
   \end{proof}
   \begin{figure}[t]
   \begin{center}
\subfigure {
\includegraphics[scale=0.58]{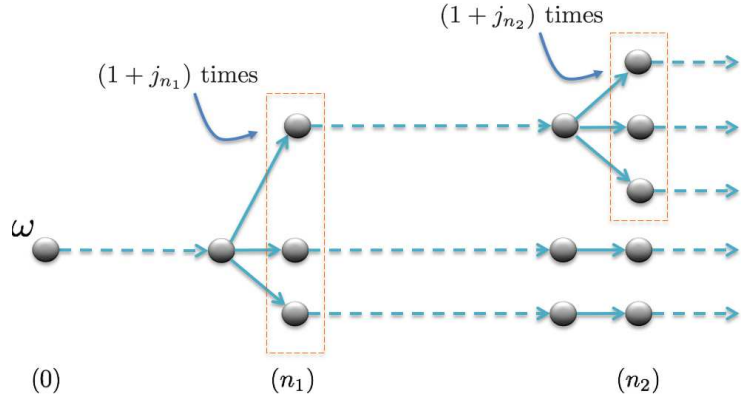}
} \caption{An example of a leafless directed tree
$\tcal$ with the properties required in the proof of
Theorem~ \ref{2isscs-t}(ii).} \label{Fig1c-general}
   \end{center}
   \end{figure}
Combining Theorem~ \ref{2isscs-t}(i), Theorem~
\ref{desz1} and Lemma~ \ref{xin11}(iv), we get the
following classification theorem.
   \begin{theorem} \label{equival}
For $k=1,2$, let $\tcal_k=(V_k,E_k)$ be a directed
tree with root $\rot_{k}$ and let $S_{\lambdab_k}\in
\B(\ell^2(V_k))$ be a $2$-isometric weighted shift on
$\tcal_k$ with weights $\lambdab_k =
\{\lambda_{k,v}\}_{v\in V_k^{\circ}}$ which satisfies
the condition \eqref{hypo+} for some
$\{\alpha_{k,v}\}_{v\in V_k} \subseteq \rbb_+$. Then
$S_{\lambdab_1} \cong S_{\lambdab_2}$ if and only if
one of the following conditions holds{\em :}
   \begin{enumerate}
   \item[(i)] $\|S_{\lambdab_1} e_{\rot_{1}}\| =
\|S_{\lambdab_2} e_{\rot_{2}}\| > 1$ and ${\mathfrak
j}^{\tcal_1}_n = {\mathfrak j}^{\tcal_2}_n$ for every
$n\in \nbb$,
   \item[(ii)] $\|S_{\lambdab_1} e_{\rot_{1}}\|
= \|S_{\lambdab_2} e_{\rot_{2}}\| =1$ and
$\sum_{n=1}^{\infty}{\mathfrak j}^{\tcal_1}_n =
\sum_{n=1}^{\infty} {\mathfrak j}^{\tcal_2}_n$.
   \end{enumerate}
   \end{theorem}
It is worth pointing out that, by \cite[Remark~ 5.8
and Lemma~ 5.9(iv)]{A-C-J-S} and Theorem~
\ref{equival}, the sequence $(\|\slam e_{\omega}\|,
{\mathfrak j}^{\tcal}_1, {\mathfrak j}^{\tcal}_2,
{\mathfrak j}^{\tcal}_3, \ldots)$ forms a complete
system of unitary invariants for non-isometric
$2$-isometric weighted shifts $\slam$ on rooted
directed trees $\tcal$ with nonzero weights satisfying
the kernel condition. In turn, the quantity
$\sum_{n=1}^{\infty} {\mathfrak j}^{\tcal}_n$ forms a
complete system of unitary invariants for isometric
weighted shifts $\slam$ on rooted directed trees
$\tcal$ (cf.\ \cite[Proposition~ 2.4]{Kub}).
   \begin{remark} \label{3uw}
Let us make a few observations concerning Theorem~
\ref{2isscs-t} (still under the assumptions of this
theorem). First, if $\slam$ is not an isometry, then
Lemma~ \ref{xin11}(iv) implies that the additive
exponent ${j}_k$ of the inflation
$\big(S_{[\xi_k(x)]}\big)^{\oplus {j}_k}$ that appears
in the orthogonal decomposition \eqref{zenob} is
maximal for every $k\in \nbb$. Second, by Lemma~
\ref{xin11}(ii), the weights of $S_{[\xi_k(x)]}$ take
the form $\{\xi_n(x)\}_{n=k}^{\infty}$. Hence, the
weights of components of the decomposition
\eqref{zenob} are built on the weights of a single
$2$-isometric unilateral weighted shift. Third, in
view of Corollary~ \ref{mulcyc} and Theorem~
\ref{2isscs-t}, general completely non-unitary
$2$-isometric operators satisfying the kernel
condition cannot be modelled by weighted shifts on
rooted direct trees. Finally, in view of Theorem~
\ref{equival}, there exist two unitarily equivalent
$2$-isometric weighted shifts on the same rooted
directed tree one with nonzero weights, the other with
some zero weights.
   \hfill $\diamondsuit$
   \end{remark}
Concluding this section, we show that there are
unitarily equivalent $2$-isometric weighted shifts on
non-graph isomorphic directed trees that satisfy
\eqref{hypo+}.
   \begin{figure}[t]
   \begin{center}
   \subfigure {
\includegraphics[scale=0.58]{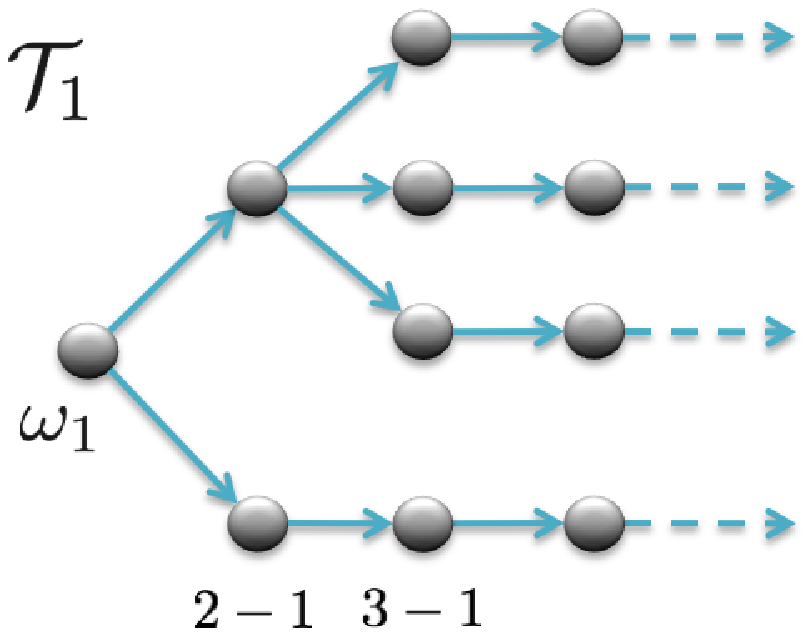}
}
   \subfigure {
\includegraphics[scale=0.58]{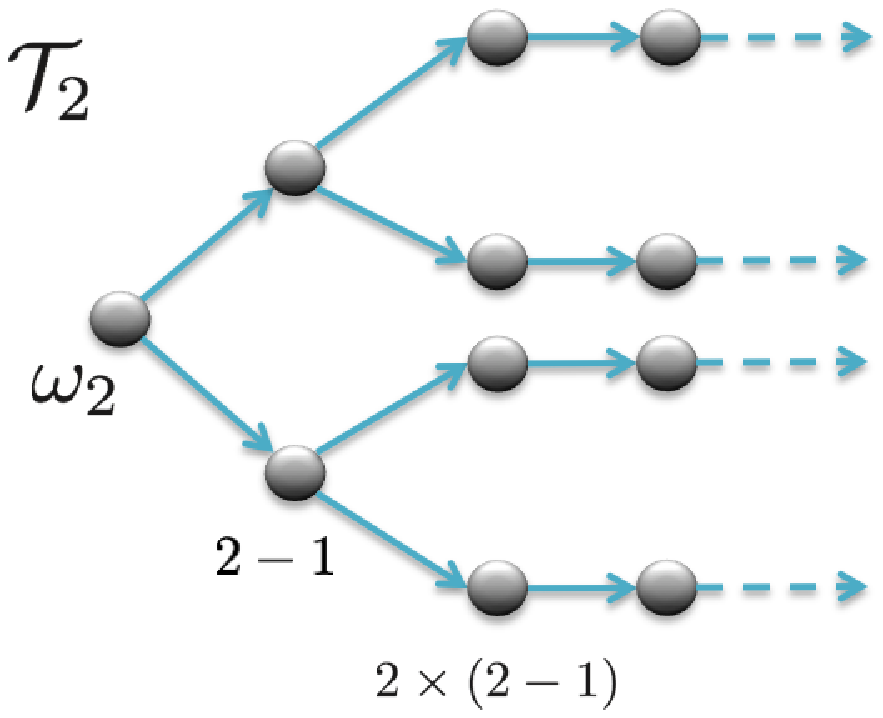}
} \caption{Two non-graph isomorphic directed trees
used in Example~ \ref{2+3}.} \label{fig1d}
   \end{center}
   \end{figure}
   \begin{example} \label{2+3}
For $k=1,2$, let $\tcal_k=(V_k,E_k)$ be a directed
tree with root $\rot_{k}$ as in Figure \ref{fig1d}.
Clearly, these two directed graphs are not graph
isomorphic. Moreover, we have (see \eqref{defjn} for
notation)
   \begin{align*}
{\mathfrak j}^{\tcal_{1}}_n = {\mathfrak
j}^{\tcal_{2}}_n =
   \begin{cases}
1 & \text{if } n=1,
   \\
2 & \text{if } n=2,
   \\
0 & \text{if } n\ge 3.
   \end{cases}
   \end{align*}
Fix $x\in [1,\infty)$. Using Procedure~ \ref{uwrem},
one can construct for $k=1,2$, a $2$-isometric
weighted shift $S_{\lambdab_k}\in \B(\ell^2(V_k))$ on
$\tcal_k$ with weights $\lambdab_k =
\{\lambda_{k,v}\}_{v\in V_k^{\circ}}$ which satisfies
the condition \eqref{hypo+} for some
$\{\alpha_{k,v}\}_{v\in V_k} \subseteq \rbb_+$ and the
equation $x=\|S_{\lambdab_k} e_{\rot_{k}}\|$. The
above combined with Theorem~ \ref{2isscs-t} implies
that
   \begin{align*}
S_{\lambdab_k} \cong S_{[x]} \oplus S_{[\xi_1(x)]}
\oplus \big(S_{[\xi_2(x)]}\big)^{\oplus 2}, \quad
k=1,2,
   \end{align*}
and so $S_{\lambdab_1} \cong S_{\lambdab_2}.$ In
particular, if $x=1$, then $S_{\lambdab_1}$ and
$S_{\lambdab_2}$ are unitarily equivalent isometries.
\hfill $\diamondsuit$
   \end{example}
   \section{\label{Sec6}The membership of the Cauchy dual operators in
$C_{0 \cdot}$ and $C_{\cdot 0}$} We begin by recalling
necessary concepts from \cite[Chapter~ II]{SF}. A
contraction $S \in \B(\hh)$ is of class $C_{0 \cdot}$
(resp., $C_{\cdot 0}$) if $S^n f \rar 0$ (resp.,
$S^{*n}f \rar 0$) as $n \rar \infty$ for all $f\in
\hh$. If $S$ is of class $C_{0 \cdot}$ and of class
$C_{\cdot 0}$, then we say that $S$ is of class
$C_{00}$. Observe that the norm of a contraction which
is not of class $C_{0 \cdot}$ (or not of class
$C_{\cdot 0}$) must equal $1$. Clearly, a contraction
$S$ is of class $C_{0 \cdot}$ if and only if $\mathsf
A_S=0,$ where $\mathsf A_S$ stands for the limit in
the strong (equivalently, weak) operator topology of
the sequence $\{S^{*n}S^{n}\}_{n=1}^{\infty}.$ That
such a limit exists plays a key role in the theory of
unitary and isometric asymptotes (see \cite[Chapter~
IX]{SF}; see also \cite[Theorem~ 1]{Dou}). As we know,
the Cauchy dual operator $T'$ of a $2$-isometry $T$ is
always a contraction (see \eqref{2hypcon}), so we can
look for an explicit description of $\mathsf A_{T'}.$
By examining the proof of \cite[Corollary~
4.6]{A-C-J-S}, we can calculate $\mathsf A_{T^\prime}$
for two classes of $2$-isometries, namely
quasi-Brownian isometries and $2$-isometries
satisfying the kernel condition. Recall that an
operator $T\in \B(\hh)$ is a {\em quasi-Brownian
isometry} if $T$ is a $2$-isometry such that
$\triangle_T T = \triangle_T^{1/2} T
\triangle_T^{1/2},$ where $\triangle_T =T^*T-I.$ A
quasi-Brownian isometry, called in \cite{Maj} a
$\triangle_T$-regular $2$-isometry, generalizes the
notion of a Brownian isometry introduced in
\cite{Ag-St}.
   \begin{lemma} \label{convcd}
Let $T\in \B(\hh)$ be a $2$-isometry and $G_T$ be the
spectral measure of $T^*T$. Then the following
assertions hold{\em :}
   \begin{enumerate}
   \item[(i)] if $T$ satisfies the kernel condition,
then $\mathsf A_{T^\prime}=G_T(\{1\}),$
   \item[(ii)] if $T$ is a quasi-Brownian isometry, then
$\mathsf A_{T^\prime}=\frac{1}{2}G_T(\{1\}) + (I +
T^*T)^{-1}.$
   \end{enumerate}
   \end{lemma}
Before stating the main result of this section, we
record the following fact.
   \begin{lemma} \label{Wiktorek}
If $T\in \B(\hh)$ is left-invertible and $T'$ is of
class $C_{0\cdot}$ or of class $C_{\cdot 0},$ then $T$
is completely non-unitary.
   \end{lemma}
   \begin{proof}
First, note the following.
   \begin{align} \label{orthsum}
   \begin{minipage}{70ex}
{\em If $T$ is left-invertible and $T$ is an
orthogonal sum of operators $A$ and $B$, i.e.,
$T=A\oplus B$, then $A$ and $B$ are left-invertible
and $T^{\prime} = A^{\prime} \oplus B^{\prime}.$}
   \end{minipage}
   \end{align}
   This together with the fact that the Cauchy dual
operator of a unitary operator is unitary completes
the proof.
   \end{proof}
Now, we can prove the main result of this section.
   \begin{theorem} \label{coo}
Let $T\in \B(\hh)$ be a $2$-isometry. Then
   \begin{enumerate}
   \item[(i)] $T'$ is of class $C_{\cdot 0}$
if and only if $T$ is completely non-unitary.
   \end{enumerate}
Moreover, if $T$ satisfies the kernel condition, then
   \begin{enumerate}
   \item[(ii)] $T'$ is of class $C_{0\cdot}$ if and only if
$T$ is completely non-unitary and $E(\{1\})=0$, where
$E$ is as in Theorem~ {\em \ref{Zak1}(iv)},
   \item[(iii)]
$T'$ is of class $C_{0\cdot}$ if and only if $T'$ is
of class $C_{00}$, or equivalently if and only if
$G_T(\{1\})=0,$ where $G_T$ is the spectral measure of
$T^*T.$
   \end{enumerate}
   \end{theorem}
   \begin{proof}
First, observe that if $T'$ is of class $C_{0 \cdot}$
or of class $C_{\cdot 0}$, then by Lemma~
\ref{Wiktorek}, $T$ is completely non-unitary. Note
also that the same conclusion holds if $G_T(\{1\})=0.$
Indeed, otherwise there exists a nonzero closed vector
subspace $\mathcal M$ of $\hh$ reducing $T$ to a
unitary operator. Then $T^*T = I$ on $\mathcal M$ and
thus $1$ is in the point spectrum of $T^*T$, which
implies that $G_T(\{1\}) \neq 0$, a contradiction.
These two observations show that there is no loss of
generality in assuming that $T$ is completely
non-unitary.

(i) It is enough to prove that $T'$ is of class
$C_{\cdot 0}$ (under the assumption that $T$ is
completely non-unitary). Using \eqref{orthsum}, the
well-known identity $(T')'=T$ (which holds for any
left-invertible operator $T$) and observing that the
Cauchy dual operator of a left-invertible normal
operator is normal and a normal $2$-isometry is
unitary (see \cite[Theorem~ 3.4]{Ja-St}), one can
deduce from \eqref{2hypcon} that $T$ is a pure and
hyponormal contraction. Since, according to
\cite[Theorem~ 3]{Put}, a pure and hyponormal
contraction is of class $C_{\cdot 0},$ we are done.

(ii)\&(iii) Assume that $T$ satisfies the kernel
condition. In view of Theorem~ \ref{Zak1}, we may
further assume that $T=W,$ where $W$ is as in (iv) of
this theorem. Using Lemma~ \ref{convcd}(i), we deduce
that $W'$ is of class $C_{0\cdot}$ if and only if
$G_W(\{1\})=0.$ We will show that
   \begin{align} \label{wprim-4}
\text{$G_W(\{1\})=0$ if and only if $E(\{1\})=0.$}
   \end{align}
Set $\eta=\sup(\supp{E})$. Note that $\eta \in [1,
\infty)$. It follows from \eqref{aopws2} and
\eqref{wagi} that
   \begin{align} \label{wprim-2}
W^{*}W = \bigoplus_{j=0}^{\infty} \int_{[1,\eta]}
\phi_{j}(x) E(dx),
   \end{align}
where $\phi_{j}\colon [1,\eta] \to \rbb_+$ is given by
$\phi_{j}(x)=\xi_j(x)^2$ for $x\in [1,\eta]$ and $j\in
\zbb_+$. By Lemma~ \ref{xin11}, $1 \Le \phi_j \Le
\eta^2$ for all $j\in \zbb_+$. This together with
\eqref{wprim-2}, \cite[Theorem~ 5.4.10]{b-s} and the
uniqueness part of the spectral theorem implies that
   \begin{align*}
G_W(\varDelta) = \bigoplus_{j=0}^{\infty}
E(\phi_{j}^{-1}(\varDelta)), \quad \varDelta \in
\borel{[1,\eta^2]}.
   \end{align*}
Since $\phi_{j}^{-1}(\{1\}) = \{1\}$ for all $j\in
\zbb_+$, we conclude that \eqref{wprim-4} holds. This
together with (i) completes the proof.
   \end{proof}
   \begin{remark}
According to \cite[Theorem~ 3.1]{Ch-1}, all positive
integer powers $T'^n$ of the Cauchy dual operator $T'$
of a $2$-hyperexpansive operator $T\in \B(\hh)$ are
hyponormal. This immediately implies that if
$T\in\B(\hh)$ is a $2$-hyperexpansive operator such
that $T'$ is of class $C_{0\cdot}$, then $T'$ is of
class $C_{0 0}.$ \hfill $\diamondsuit$
   \end{remark}
Regarding Theorem~ \ref{coo}, note that there exist
completely non-unitary $2$-isome\-tries satisfying the
kernel condition whose Cauchy dual operators are not
of class $C_{0\cdot}.$ To see this, consider a nonzero
Hilbert space $\mathcal M$ and a compactly supported
$\B(\mathcal M)$-valued Borel spectral measure $E$ on
the interval $[1,\infty)$ such that $E(\{1\}) \neq 0$.
Then, by Theorems~ \ref{Zak1} and \ref{coo}(ii), the
operator valued unilateral weighted shift $W$ on
$\ell^2_{\mathcal M}$ with weights
$\{W_n\}_{n=0}^{\infty}$ defined by \eqref{wagi} has
all the required properties.

The following proposition shows that unlike the case
of $2$-isometries satisfying the kernel condition, the
Cauchy dual operator of a quasi-Brownian isometry is
never of class $C_{0.}$ (see also Lemma~
\ref{convcd}(ii)).
   \begin{proposition} \label{coo-qB} Let $T\in
\B(\hh)$ be a $2$-isometry and let $T'$ be its Cauchy
dual operator. Then the following assertions hold{\em
:}
   \begin{enumerate}
   \item[(i)] if $T$ is a quasi-Brownian
isometry, then for every $n\in \zbb_+$,
   \begin{align} \label{coo-qBw}
\|T^{\prime n} f\|^2 \Ge c_n \|f\|^2, \quad f\in \hh,
   \end{align}
where $c_n=\frac{1+\|T\|^{2(1-2n)}}{1+\|T\|^2}$ is the
largest constant for which \eqref{coo-qBw} holds; in
particular, $T'$ is not of class $C_{0.}$ and
$\|T'\|=1,$
   \item[(ii)] if $T$ satisfies
the kernel condition, then for every $n\in \zbb_+$,
   \begin{align} \label{coo-kcw}
\|T^{\prime n} f\|^2 \Ge c_n \|f\|^2, \quad f\in \hh,
   \end{align}
where $c_n=\frac{1}{1+n(\|T\|^2-1)}$ is the largest
constant for which \eqref{coo-kcw} holds.
   \end{enumerate}
   \end{proposition}
   \begin{proof}
(i) Fix $n\in \zbb_+.$ Note that $T^{\prime n}$ is
left-invertible. Denote by $\hat c_n$ the largest
positive constant for which \eqref{coo-qBw} holds.
Define $s_n\colon [1,\infty) \to (0,\infty)$ by
   \begin{align*}
s_n(x) = \frac{1+x}{1+{x^{1-2n}}}, \quad x \in [1,
\infty).
   \end{align*}
Using \cite[Theorem~ 4.5]{A-C-J-S}, the fact that
$\sigma(T^*T) \subseteq [1,\infty)$ and the functional
calculus (see \cite[Theorem~ VIII.2.6]{Con}), we
deduce that
   \begin{gather*}
\hat c_n = \frac{1}{\|(T'^{*n}T'^n)^{-1}\|} =
\frac{1}{\|s_n(T^*T)\|} = \frac{1}{\sup_{x\in
\sigma(T^*T)}s_n(x)}
   \\
= \frac{1}{s_n(\sup \sigma(T^*T))} =
\frac{1}{s_n(\|T\|^2)}.
   \end{gather*}
Due to \eqref{2hypcon}, the ``in particular'' part of
(i) is now clear.

(ii) Argue as in (i) using \cite[Theorem~
3.3]{A-C-J-S} in place of \cite[Theorem~
4.5]{A-C-J-S}.
   \end{proof}
As a direct consequence of Proposition~ \ref{coo-qB}
and the fact that $\|T\|\Ge 1$ for any $2$-isometry
$T$ (see \cite[Lemma~ 1]{R-0}), we get
   \begin{align*}
\lim_{n\to \infty} c_n =
   \begin{cases}
0 & \text{if $T$ is a $2$-isometry satisfying
\eqref{kc} and } \|T\|\neq 1,
   \\
\frac{1}{1+\|T\|^2} & \text{if $T$ is a quasi-Brownian
isometry and } \|T\|\neq 1.
   \end{cases}
   \end{align*}
{\bf Acknowledgments.} A part of this paper was
written while the second author visited Jagiellonian
University in Summer of 2018. He wishes to thank the
faculty and the administration of this unit for their
warm hospitality.
   \bibliographystyle{amsplain}
   
   \end{document}